\documentclass[11pt,leqno]{amsart}
\input{article.sty}
\usepackage{comment}
\usepackage{thmtools, thm-restate}

\author{Salvatore Lincastri}
\address{Salvatore Lincastri
  \newline \indent Dipartimento di Matematica e Applicazioni
  \newline \indent
  Università degli Studi di Milano-Bicocca
\newline\indent Via Roberto Cozzi 55, 20125, Milano, Italy.}
\email{s.lincastri@campus.unimib.it}

\date{}

\title{Gradient estimates and Liouville properties for the drifted Laplacian}

\begin{document}
\begin{abstract}
In this paper, we discuss the validity of Liouville-type theorems involving the drifted Laplacian on complete, non-compact Riemannian manifolds endowed with a smooth vector field $X$. In particular, we prove new Liouville properties both for positive $X$-harmonic functions and for non-negative $X$-subharmonic functions which are in an $L^{p}$-space, $p\in(1,+\infty)$, with respect to any invariant measure for the operator.

The main tool for the analysis is a local gradient estimate for positive solutions to the semilinear equation $\Delta_{X}u=F(u)$, which is valid if $F$ satisfies certain structural conditions and the Bakry–Émery-Ricci curvature associated to $X$ is bounded below. On the one hand, this estimate is used in a novel way to deduce the Liouville property for positive solutions and, on the other hand, it is used to construct cut-off functions for the $X$-Laplacian that yields the $L^{p}$-Liouville results.
\end{abstract}

\maketitle

\section{Introduction} 
Liouville-type theorems constitute a classical and important topic in the study of elliptic partial differential equations on Riemannian manifolds. Roughly speaking, such results assert that, under geometric and analytic hypotheses on the underlying manifold, solutions of differential inequalities satisfying suitable global assumptions must be trivial or, in appropriate settings, constant. The assumptions imposed on the solution are usually of two kinds: either pointwise controls (boundedness, positivity, prescribed growth rates, ...) or $L^p$-integrability for some $p\in(0,+\infty)$. Depending on this choice, such results reveal deep connections between the global geometry of the underlying space and the behavior of solutions to elliptic equations (for instance, see \cite{PigolaRimoldiSetti} for triviality results for expanding Ricci solitons under $L^{p}$ conditions, $p\in[1,+\infty]$, \cite{Grigor1999} for the relation between Liouville properties, parabolicity and stochastic completeness, and \cite{PVV} and the references therein for the relation between $L^p$-Liouville properties for subharmonic distributions and functional analytic properties of Schrödinger operators). 

In the case of harmonic functions on Riemannian manifolds, the pioneering works of Yau and Cheng-Yau (\cite{yau1975harmonic,cheng1975differential}) 
establish the Liouville theorem for positive harmonic functions on complete manifolds with non-negative Ricci curvature, whereas Anderson and Sullivan (\cite{anderson1983dirichlet,Sullivan}) showed that it fails in some negatively curved settings, namely Cartan–Hadamard manifolds with pinched curvature. On the other hand, Yau proved in \cite{yau1976some} the Liouville property for smooth, non-negative subharmonic functions in $L^{p}(M)$, $p\in(1,+\infty)$, on general complete Riemannian manifolds (see also \cite{PVV} for the case of subharmonic distributions).

A natural generalization of the previous results is in the context of $f$-harmonic functions on metric measure spaces. Consider a smooth metric measure space $(M,g,\mathrm{dv}_f:=e^{-f}\mathrm{dv})$, where dv denotes the Riemannian volume measure and $f\in C^{\infty}(M)$. In this setting, the Laplace–Beltrami operator has a natural weighted analogue, called the $f$-Laplacian, defined by
\begin{equation*}
\Delta_{f}:=\Delta-g(\nabla f,\nabla\cdot)
=e^{f}\mathrm{div}(e^{-f}\nabla\cdot),
\end{equation*}
which is a symmetric operator with respect to the weighted measure dv$_{f}$. The solutions to the equation
\begin{equation*}
\Delta_{f}u=0
\end{equation*}
are called $f$-harmonic functions. Notice that, when $f$ is constant, one recovers the usual notion of a harmonic function.
\\
In this setting, the natural notions of curvature are the $N$-Bakry–Émery Ricci curvature, $N>0$, defined as
\begin{equation*}
\mathrm{Ric}_{f}^{N}
:=\mathrm{Ric}+\mathrm{Hess}(f)
-\frac{1}{N}\mathrm{d}f\otimes\mathrm{d}f,
\end{equation*}
and the $\infty$-Bakry–Émery Ricci curvature, also simply called the Bakry–Émery Ricci curvature, defined as
\begin{equation*}
\mathrm{Ric}_{f}
:=\mathrm{Ric}+\mathrm{Hess}(f),
\end{equation*}
where $\mathrm{Ric}$ and $\mathrm{Hess}(f)$ denote, respectively, the Ricci curvature of $M$ and the Hessian of $f$. In the case of metric measure spaces, the situation regarding the weighted versions of the aforementioned Liouville theorems is fairly well understood: see \cite{li2005liouville, mastrolia2010diffusion, brighton2013liouville, munteanu2011smooth} for positive $f$-harmonic functions, and \cite{PigolaRS-Vanishing, PigolaRimoldiSetti, WuWu2014p} for non-negative subharmonic functions in $L^{p}(M,\mathrm{dv}_f)$. 

When the drift is no longer a gradient, the situation is much less well understood. In this setting the metric measure structure is lost, but a suitable generalization of the Laplacian and of the associated Ricci curvature is still available. Consider a smooth vector field $X$ on a complete, non-compact Riemannian manifold $(M,g)$. The \emph{drifted} \emph{Laplacian}, or $X$-\emph{Laplacian}, is the elliptic differential operator given by 
\begin{equation*}
    \Delta_{X}:=\Delta-g(X,\nabla\cdot),
\end{equation*}
and an $X$-harmonic function is a solution to the equation
\begin{equation*}
    \Delta_{X}u=0.
\end{equation*} 
The curvature conditions related to the $X$-Laplacian are contained in a family of tensors called $N$-Bakry–Émery-Ricci curvatures associated to $X$, $N>0$, defined as
\begin{equation*}
    \mathrm{Ric}^{N}_{X}:=\mathrm{Ric}+\frac{1}{2}\mathcal{L}_{X}g-\frac{1}{N}X^{\flat}\otimes X^{\flat},
\end{equation*}
where $\mathcal{L}_{X}g$ denotes the Lie derivative of the metric in the direction of $X$ and $\cdot^{\flat}$ is the musical isomorphism. When $N=+\infty$, we obtain the Bakry–Émery-Ricci curvature, simply denoted by
\begin{equation*}
    \mathrm{Ric}_{X}:=\mathrm{Ric}+\frac{1}{2}\mathcal{L}_{X}g.
\end{equation*}
Note that if $X=\nabla f$ for some smooth function $f$ on $M$, we recover the notion of $f$-harmonic functions, and the ($N$-)Bakry–Émery Ricci curvature of $X$ reduces to the usual one, $\mathrm{Ric}+\mathrm{Hess}(f)$ (with the corresponding $N$-term), by the well-known identity
\begin{equation*}
    \frac{1}{2}\mathcal{L}_{\nabla f}g = \mathrm{Hess}(f).
\end{equation*}  

The aim of this paper is to establish conditions on $X$ that guarantee the validity of the Liouville property for positive $X$-harmonic functions when the Bakry–Émery Ricci curvature is non-negative, as well as for non-negative $X$-subharmonic functions satisfying $L^{p}$-integrability conditions with respect to some fixed measure.

Concerning the first property, we begin by extending a result by Munteanu-Wang (\cite{munteanu2011smooth}) to the non-gradient case. Namely, they proved that on a complete metric measure space $(M,g,\mathrm{dv}_f)$ with Ric$_{f}\ge0$, any positive $f$-harmonic function $u$ satisfying the sub-exponential growth condition
\begin{equation*}
    \limsup_{R\rightarrow+\infty}\frac{\sup_{B_{R}(o)}\log(u+1)}{R}=0
\end{equation*}
must be constant, and the assumption can be removed if the potential function $f$ has sublinear growth. Note that, in the latter case, it happens that 
\begin{equation}\label{sublinearity_grad_f}
    \liminf_{x\rightarrow\infty}\vert\nabla f\vert=0.
\end{equation} 
If we consider a non-gradient vector field $X$ instead of $\nabla f$, the result regarding positive $X$-harmonic functions of sub-exponential growth can be recovered by combining the techniques in \cite{munteanu2011smooth} with a comparison principle in \cite{catino2016analytic}. However, a condition similar to \eqref{sublinearity_grad_f} on the norm of $X$ is not enough to guarantee the strong Liouville property (see Remark \ref{optimality_Liou_positive_decaying}), and the techniques in \cite{munteanu2011smooth} cannot be applied. In this setting, the right condition to impose is that the norm of the vector field $X$ tends to zero at infinity. The first result of this paper is the following
\begin{theorem}
    Let $(M,g)$ be a complete, non-compact Riemannian manifold endowed with a smooth vector field $X$. If $\mathrm{Ric}_{X}\ge 0$, then any positive $X$-harmonic function satisfying
    \begin{equation}\label{intro_subexp_growth1}
       \limsup_{R\rightarrow+\infty}\frac{\sup_{B_{R}(o)}\log(u+1)}{R}=0
    \end{equation}
    must be constant. Moreover, assumption \eqref{intro_subexp_growth1} can be removed if there exists a point $o\in M$ such that
    \begin{equation*}
        \vert X\vert\le \Lambda(r(x)) \quad \forall x\in M,
    \end{equation*}
    where $r(x)$ is the distance from $o$ and $\Lambda: [0,+\infty)\longrightarrow [0,+\infty)$ a continuous function satisfying
    \begin{equation*}
        \lim_{t\rightarrow+\infty}\Lambda(t) = 0.
    \end{equation*}
\end{theorem}
\noindent The proof combines a new Cheng–Yau type gradient estimate for $\Delta_X$ with the maximum principle, applied to a suitable function built from the positive $X$-harmonic function.

Considerably less is known about $L^{p}$-Liouville properties, and many difficulties arise because of the lack of symmetry of the operator $\Delta_{X}$ with respect to the Riemannian measure. The property fails even in simple cases when the reference measure is the standard Riemannian measure, as shown in Example \ref{failure_LpLiou_notInv}. Nevertheless, adopting a probabilistic viewpoint, Pessoa, Pigola and Setti identified in \cite{PessoaPigolaSetti2026volumegrowthparabolicitynonsymmetric} the appropriate framework in which many aspects of the analysis of the weighted Laplacian can be extended to the non-symmetric case, namely that of special invariant measures, which display a somewhat canonical behavior. 
\\
Following this point of view, the second main result of this paper is the following $L^p$-Liouville theorem, which is, in some sense, complementary to Theorem 6.3 in \cite{PessoaPigolaSetti2026volumegrowthparabolicitynonsymmetric}.
\begin{theorem}
    Let $(M,g)$ be a complete Riemannian manifold and $p\in(1,+\infty)$. Let $X$ be a smooth vector field such that $\vert X\vert\le\Lambda$ and $\mathrm{Ric}_{X}\ge-(n-1)K$ for some $\Lambda, K\ge0$. If $\mu$ is a positive Radon measure which is an invariant measure for $\Delta_{X}$, then any function $ u\in C^{\infty}(M)$ satisfying
    \begin{equation*}
        \begin{cases}
            \Delta_{X}u\ge 0,\\
            0\le u\in L^{p}(M,\mu)
        \end{cases}
    \end{equation*}
    is constant. 
\end{theorem}
\noindent This Liouville theorem is a consequence of the existence of a special class of cutoff functions, which we call $X$-Laplacian cutoff functions (in analogy with a similar notion in the case of the Laplacian). Precisely, we prove the following
\begin{theorem}
    Let $(M,g)$ be a complete Riemannian manifold, fix a point $o\in M$ and denote by $B_{R}$ the ball of center $o$ and radius $R$. Let $X$ be a smooth vector field such that $\vert X\vert\ \le\Lambda$ and $\mathrm{Ric}_{X}\ge-(n-1)K$ for some $\Lambda, K\ge0$. For every $R\ge2$ there exist $\varphi_{R}\in C^{\infty}_{c}(M)$ and constants $C>0$ and $\gamma>1$ such that
    \begin{enumerate}
        \item $0\le\varphi_{R}\le1$,
        \item $\varphi_{R}\equiv1$ on $B_{R}$,
        \item $\varphi_{R}\equiv 0$ on $M\setminus B_{\gamma R}$,
        \item $\vert \nabla\varphi_{R}\vert\le\frac{C}{R}$,
        \item $\vert \Delta_{X}\varphi_{R}\vert\le\frac{C}{R}$.
    \end{enumerate}
\end{theorem}
\noindent In the case of the standard Laplacian, the existence of this type of cutoff functions is well understood (see \cite{guneysu2016sequences, BS18}) and has several applications: for instance, it is related to the density of the space of compactly supported smooth functions in some Sobolev spaces and to the so-called $L^{p}$-positivity preserving property, which has consequences on the realization of Schrödinger operators (see \cite{guneysu2016sequences, BS18, BMS, MV24}).

The main tool for the analysis of the aforementioned Liouville properties will be the following gradient estimate for positive solutions to the semilinear equation $\Delta_{X}u=F(u)$. 
\begin{theorem}
    Let $(M,g)$ be a complete Riemannian manifold of dimension $n$, and let $X$ be a smooth vector field on $M$. Having fixed a point $o\in M$ and $R>0$, assume that $\vert X\vert \le \Lambda$ and $\mathrm{Ric}_{X}\ge-(n-1)K$ on the ball $B_{2R}(o)$ for some constants $K,\Lambda\ge0$. Let $u\in C^{\infty}(B_{2R}(o))$ be a positive solution to
    \begin{equation*}
        \Delta_{X}u=F(u) \quad \text{on $B_{2R}(o)$},
    \end{equation*}
    where $F\in C^{\infty}([0,+\infty))$ satisfies the structural conditions 
    \begin{equation*} 
        F(t) - tF'(t) \le \alpha t, \quad \quad \vert F(t)\vert \le\beta t, 
    \end{equation*}    
    for some $\alpha\in\mathbb{R}$ and $\beta \ge 0$. Then, for every $x\in B_{R}(o)$ we have the estimate
    \begin{equation*}
         \frac{\vert\nabla u\vert^{2}}{u^{2}}(x) \le C(n)\bigg(\max\bigg\{\alpha+ \frac{3}{2}(n-1)K,0\bigg\} + \beta + \Lambda^2 + \frac{1}{R^{2}}\bigg),
    \end{equation*}
    where $C(n)>0$ is a constant depending only on $n$.
\end{theorem}
\noindent Estimates of this kind have a long history. Starting from the pioneering works of Yau and Cheng–Yau (\cite{yau1975harmonic}, \cite{cheng1975differential}), in which the maximum principle method was used to obtain gradient estimates for positive harmonic functions, many authors have adapted their technique to different operators and equations. For examples of gradient estimates and their applications, for both weighted and unweighted operators, we refer to \cite{brighton2013liouville, Li-ChenGradientActa, li2005liouville, ma2018gradient, mastrolia2010diffusion, munteanu2011smooth, li1986parabolic, li1979lower, li1980estimates}.

In the context of $X$-harmonic functions (corresponding to the case $\alpha=\beta=0$), a gradient estimate of this type was found by Gonz\'ales and Negr\'in in \cite{gonzalez1999gradient}. There, the authors employed the maximum principle method to establish a gradient estimate for positive solutions of the parabolic equation associated to the $X$-Laplacian (see also \cite{cranston1991gradient}, \cite{thalmaier1998gradient} for Schauder-type gradient estimates). Their proof requires a lower bound on the Riemannian Ricci curvature and an upper bound on its covariant derivative $\nabla X$ on a ball $B_{2R}(o)$ of center $o\in M$ and radius $2R$. Here, we extend their estimate by requiring only a lower bound on the Bakry–Émery-Ricci curvature, which is implied by the combination of the bounds on the Riemannian Ricci curvature of $M$ and on the covariant derivative of $X$. 

The paper is organized as follows. In Section 2, we recall the Bochner formula for the $X$-Laplacian, and we prove a local comparison theorem for the $X$-Laplacian of the distance function. In Section 3, we prove the gradient estimate and derive some natural consequences, such as the Harnack inequality. In Section 4, we discuss the validity of the Liouville property for positive $X$-harmonic functions in the context of non-negative Bakry–Émery-Ricci curvature, and we prove our first Liouville theorem. Finally, in Section 5 we use the gradient estimate to prove the existence of $X$-Laplacian cutoff functions, and we establish the Liouville property for non-negative $X$-subharmonic functions which belong to $L^{p}$ with respect to a general invariant measure.

\section{Preliminary results}
In this section, we introduce the main tools for the proof of the gradient estimate.
Let $(M,g)$ be a complete, non-compact Riemannian manifold of dimension $n$, and let $X$ be a smooth vector field on $M$. The first ingredient we need is a version of the Bochner formula for the $X$-Laplacian, which is the content of the following lemma.
\begin{lemma}\label{X-Bochner}
    Let $(M,g)$ be a Riemannian manifold, let $X$ be a smooth vector field on $M$, and let $u\in C^{3}(M)$. Then
    \begin{equation}\label{formula X-Bochner}
        \frac{1}{2}\Delta_{X}\vert\nabla u\vert^{2} = \vert \mathrm{Hess}(u)\vert^{2}+g(\nabla u,\nabla\Delta_{X}u)+\mathrm{Ric}_{X}(\nabla u, \nabla u).
    \end{equation}
\end{lemma}
\begin{proof}
    This formula is standard, but we provide a proof for the sake of completeness. By the classical Bochner formula we have
    \begin{equation*}
        \frac{1}{2}\Delta\vert\nabla u\vert^{2}=\vert \mathrm{Hess}(u)\vert^{2}+g(\nabla u,\nabla\Delta u)+\mathrm{Ric}(\nabla u,\nabla u),
    \end{equation*}
    and applying the definition of $\Delta_{X}$ and $\mathrm{Ric}_{X}$ we get
    \begin{align*}
        \frac{1}{2}\Delta_{X}\vert\nabla u\vert^{2}+\frac{1}{2}g(X,\nabla\vert\nabla u\vert^{2})&=\vert \mathrm{Hess}(u)\vert^{2}+g(\nabla u,\nabla\Delta_{X} u)\\
        &+g(\nabla u,\nabla g(X,\nabla u))+\mathrm{Ric}_{X}(\nabla u,\nabla u)-\frac{1}{2}\mathcal{L}_{X}g(\nabla u,\nabla u).
    \end{align*}
    To prove (\ref{formula X-Bochner}), we show that
    \begin{equation*}
        g(\nabla u,\nabla g(X,\nabla u))-\frac{1}{2}\mathcal{L}_{X}g(\nabla u,\nabla u)-\frac{1}{2}g(X,\nabla\vert\nabla u\vert^{2})=0.
    \end{equation*}
    Fix an orthonormal frame $\{e_{i}\}_{i=1,\dots,n}$ and let $\{\theta^{i}\}_{i=1,\dots,n}$ be the dual orthonormal coframe. On these frames, the components of the gradient of $g(X,\nabla u)$ are
    \begin{equation*}
        (g(X,\nabla u))_{i}=(X^{j}u_{j})_{i}=X^{j}_{i}u_{j}+X^{j}u_{ij},
    \end{equation*}
    where $u_{i}$, $u_{ij}$, $X^{i}$ and $X_{j}^{i}$ are the components, respectively, of $\nabla u$, $\mathrm{Hess}(u)$, $X$ and $\nabla X$ (note that we are using the Einstein summation convention).
    Since
    \begin{equation*}
        \mathcal{L}_{X}g(Y,Z) = (X^{i}_{j}+X^{j}_{i})Y^{i}Z^{j}
    \end{equation*}
    for every pair of vector fields $Y$ and $Z$, we find
    \begin{align*}
        g(\nabla u,\nabla(g(X,\nabla u)))&=u_{i}(X^{j}_{i}u_{j}+X^{j}u_{ij})=\frac{1}{2}(X^{i}_{j}+X^{j}_{i})u_{i}u_{j}+u_{ij}X^{j}u_{i} \\
        &=\frac{1}{2}\mathcal{L}_{X}g(\nabla u,\nabla u)+\frac{1}{2}X^{j}(u^{2}_{i})_{j} \\
        &=\frac{1}{2}\mathcal{L}_{X}g(\nabla u,\nabla u)+\frac{1}{2}g(X,\nabla\vert\nabla u\vert^{2}),
    \end{align*}
    which is exactly the identity we wanted to prove.
\end{proof}
 The second result we need is a local comparison theorem for the $X$-Laplacian of the distance function. For the proof, we  localize an argument due to Qian (see \cite{qian1996conservation}) and use the appropriate test function in the estimates. Let $K\ge0$, and recall that $\mathrm{sn}_{-K}(t)$ is the function
    \begin{equation*}
        \mathrm{sn}_{-K}(t):=
        \begin{cases}
            t &\text{if $K=0$}, \\
            \frac{1}{\sqrt{K}}\sinh(\sqrt{K}t) &\text{if $K>0$}.
        \end{cases}
    \end{equation*}
We then obtain the following comparison theorem.
\begin{theorem}\label{X-comparison for vector field}
    Let $(M,g)$ be a complete, non-compact Riemannian manifold of dimension $n$, and let $X$ be a smooth vector field on $M$. Fix a point $o\in M$ and denote by $r$ the distance function from $o$. If $\mathrm{Ric}_{X}\ge -(n-1)K$ and $\vert X\vert \le \Lambda$ on $B_{R}(o)$ for some constants $K,\Lambda\ge0$, then we have the estimate
    \begin{equation}
        \Delta_{X}r \le (n-1)\frac{\textnormal{sn}_{-K}'(r)}{\textnormal{sn}_{-K}(r)} + \Lambda \quad\forall x\in B_{R}(o)\setminus (cut(o)\cup\{o\}),
    \end{equation}
    where $cut(o)$ denotes the cut locus of $o$.
\end{theorem}
\begin{proof}
    Having fixed a point $x\in B_{R}(o)\setminus(cut(o)\cup\{o\})$, set $\ell:=r(x)$ and let $\gamma: [0,\ell]\longrightarrow M$ be a minimizing geodesic connecting $o$ and $x$, parametrized by $g$-arclength. By the Gauss lemma, we have $\nabla r(\gamma(t))=\dot{\gamma}(t)$ for every $t\in(0,\ell]$, and a simple computation shows that
    \begin{equation*}
        \frac{\textnormal{d}}{\textnormal{d}t}g_{\gamma(t)}(X,\nabla r)=\frac{1}{2}\mathcal{L}_{X}g(\dot{\gamma}(t),\dot{\gamma}(t)).
    \end{equation*}
    The first step is to obtain an estimate for the Laplacian of $r$ at $x$. Fix $\delta\in(0,\ell)$ and let $h\in C^{0,1}((\delta,\ell])$ be a non-negative Lipschitz function satisfying $h(\delta)=0$. Applying the Bochner formula to $r$ and restricting along $\gamma(t)$, we find %ricorda che hess(r)(\nabla r,.)=0 -> norma e norma di una matrice n-1 per n-1
    \begin{align*}
        0 &= \vert \mathrm{Hess}(r)\vert^{2}+g(\nabla r,\nabla\Delta r) + \mathrm{Ric}(\nabla r,\nabla r) \\
        &\ge \frac{(\Delta r \circ  \gamma(t))^{2}}{n-1} + \frac{\textnormal{d}}{\textnormal{d}t}(\Delta r \circ \gamma(t))+\mathrm{Ric}(\dot{\gamma}(t),\dot{\gamma}(t)),
    \end{align*}
    hence
    \begin{equation*}
          \frac{\textnormal{d}}{\textnormal{d}t}(\Delta r \circ \gamma(t)) + \frac{(\Delta r \circ  \gamma(t))^{2}}{n-1}\le -\mathrm{Ric}(\dot{\gamma}(t),\dot{\gamma}(t)).
    \end{equation*}
    Multiplying this equation by $h^{2}(t)$ and integrating from $\delta$ to $\ell$, we obtain
    \begin{equation*}
        \int_{\delta}^{\ell}h^{2}(t)\frac{\textnormal{d}}{\textnormal{d}t}(\Delta r \circ \gamma(t))\ \textnormal{d}t + \int_{\delta}^{\ell} h^{2}(t)\frac{(\Delta r \circ  \gamma(t))^{2}}{n-1} \textnormal{d}t \le -\int_{\delta}^{\ell} h^{2}(t)\mathrm{Ric}(\dot{\gamma}(t),\dot{\gamma}(t))\ \textnormal{d}t,
    \end{equation*}
    and after an integration by parts we find
    \begin{equation*}
        h^{2}(\ell)\Delta r(x) -\int_{\delta}^{\ell}2h'(t)h(t)(\Delta r \circ \gamma(t))\ \textnormal{d}t + \int_{\delta}^{\ell} h^{2}(t)\frac{(\Delta r \circ  \gamma(t))^{2}}{n-1} \textnormal{d}t \le -\int_{\delta}^{\ell} h^{2}(t)\mathrm{Ric}(\dot{\gamma}(t),\dot{\gamma}(t))\ \textnormal{d}t.
    \end{equation*}
    From Young's inequality, we have
    \begin{equation*}
        2h'(t)h(t)(\Delta r \circ \gamma(t)) \le (n-1)(h'(t))^{2}+h^{2}(t)\frac{(\Delta r \circ \gamma(t))^{2}}{n-1},
    \end{equation*}
    and inserting this in the previous formula we obtain
    \begin{equation}\label{equazione (2) comparison}
        h^{2}(\ell)\Delta r(x) \le \int_{\delta}^{\ell}(n-1)(h'(t))^{2}\ \textnormal{d}t-\int_{\delta}^{\ell} h^{2}(t)\mathrm{Ric}(\dot{\gamma}(t),\dot{\gamma}(t))\ \textnormal{d}t.
    \end{equation}
    Inequality (\ref{equazione (2) comparison}) together with the definition of $\mathrm{Ric}_{X}$ then imply the following estimate for the Laplacian of $r$:
    \begin{align*}
        h^{2}(\ell)\Delta r(x) &\le \int_{\delta}^{\ell}(n-1)(h'(t))^{2}\ \textnormal{d}t-\int_{\delta}^{\ell} h^{2}(t)\mathrm{Ric}_{X}(\dot{\gamma}(t),\dot{\gamma}(t))\ \textnormal{d}t+\int_{\delta}^{\ell}\frac{1}{2}h^{2}(t)\mathcal{L}_{X}g(\dot{\gamma}(t),\dot{\gamma}(t))\ \textnormal{d}t \\
        &= \int_{\delta}^{\ell}(n-1)(h'(t))^{2}\ \textnormal{d}t-\int_{\delta}^{\ell} h^{2}(t)\mathrm{Ric}_{X}(\dot{\gamma}(t),\dot{\gamma}(t))\ \textnormal{d}t +\int_{\delta}^{\ell}h^{2}(t)\frac{\textnormal{d}}{\textnormal{d}t}g_{\gamma(t)}(X,\dot{\gamma}(t))\ \textnormal{d}t \\
        &= \int_{\delta}^{\ell}(n-1)(h'(t))^{2}\ \textnormal{d}t-\int_{\delta}^{\ell} h^{2}(t)\mathrm{Ric}_{X}(\dot{\gamma}(t),\dot{\gamma}(t))\ \textnormal{d}t + h^{2}(\ell)g_{x}(X,\nabla r)\\
        &-\int_{\delta}^{\ell}2h'(t)h(t)g_{\gamma(t)}(X,\dot{\gamma}(t))\ \textnormal{d}t.
    \end{align*}
    By definition of $\Delta_{X}$ and exploiting the bound on $\mathrm{Ric}_{X}$, we get
    \begin{equation}\label{equazione (4) comparison}
    \begin{split}
        h^{2}(\ell)\Delta_{X} r(x) &\le \int_{\delta}^{\ell}(n-1)(h'(t))^{2}\ \textnormal{d}t-\int_{\delta}^{\ell} h^{2}(t)\mathrm{Ric}_{X}(\dot{\gamma}(t),\dot{\gamma}(t))\ \textnormal{d}t \\
        &-\int_{\delta}^{\ell}2h'(t)h(t)g_{\gamma(t)}(X,\nabla r)\ \textnormal{d}t \\
        & \le (n-1)\int_{\delta}^{\ell}[(h'(t))^{2} + K h^{2}(t)]\ \textnormal{d}t+2\Lambda\int_{\delta}^{\ell}\vert h'(t)h(t)\vert \ \textnormal{d}t.
    \end{split}
    \end{equation}
    If we choose $h(t):=\textnormal{sn}_{-K}(t)-\textnormal{sn}_{-K}(\delta)$ as the test function in (\ref{equazione (4) comparison}), and if we send $\delta$ to 0, we find
    \begin{equation*}
    \begin{split}
        \Delta_{X} r(x) \le (n-1)\frac{\textnormal{sn}_{-K}'(r(x))}{\textnormal{sn}_{-K}(r(x))}+\Lambda,
    \end{split}
    \end{equation*}
    which is the desired estimate.
\end{proof}

\begin{remark}
    When the vector field $X$ is bounded, Theorem \ref{X-comparison for vector field} gives a global comparison theorem, and the inequality becomes
    \begin{equation*}
        \Delta_{X}r \le (n-1)\frac{\textnormal{sn}_{-K}'(r)}{\textnormal{sn}_{-K}(r)} + \Lambda \quad \forall x\in M\setminus (cut(o)\cup\{o\}).
    \end{equation*}
    Moreover, arguing as in the proof of Proposition 2.2 in \cite{catino2016analytic} one can show that the estimate holds also in the distributional sense on $M$.
\end{remark}
\begin{remark}
It is standard to see that the estimate in the comparison theorem implies 
    \begin{equation}\label{X-comparison estimate}
        \Delta_{X}r \le \frac{n-1}{r}+(n-1)\sqrt{K} + \Lambda
    \end{equation}
    for every point $x\in B_{R}(o)\setminus (cut(o)\cup\{o\})$. This is the form we shall use in the next sections.
\end{remark}

\section{Gradient estimate for the $X$-Laplacian}
Let $(M,g)$ be a complete Riemannian manifold, and let $X$ be a smooth vector field on $M$. In this section, we prove a gradient estimate for positive solutions to the semilinear equation
\begin{equation}\label{semilinear X-lap}
    \Delta_{X}u=F(u),
\end{equation}
where $F\in C^{\infty}([0,+\infty))$ satisfies the structural conditions 
    \begin{align} 
        &F(t) - tF'(t)\le \alpha t, \label{struttura 1 F} \\ 
        &\vert F(t)\vert \le\beta t, \label{struttura 2 F}
    \end{align}
for some $\alpha\in\mathbb{R}$ and $\beta \ge 0$. The proof relies on the maximum principle method, and makes use of the $X$-Bochner formula and the $X$-Laplacian comparison theorem that we introduced in the previous section.
\begin{theorem}\label{general gradient estimate for X-Lap}
    Let $(M,g)$ be a complete Riemannian manifold of dimension $n$, and let $X$ be a smooth vector field on $M$. Let $o\in M$ be a fixed point, let $R>0$, and assume that $\mathrm{Ric}_{X}\ge-(n-1)K$ and $\vert X\vert \le \Lambda$ on the ball $B_{2R}(o)$, for some constants $K,\Lambda\ge0$. Let $u\in C^{\infty}(B_{2R}(o))$ be a positive solution to
    \begin{equation*}
        \Delta_{X}u=F(u) \quad \text{on $B_{2R}(o)$},
    \end{equation*}
    where $F\in C^{\infty}([0,+\infty))$ satisfies the structural conditions (\ref{struttura 1 F}) and (\ref{struttura 2 F})    
    for some $\alpha\in\mathbb{R}$ and $\beta \ge 0$. Then, for every $x\in B_{R}(o)$ we have the estimate
    \begin{equation}\label{gradient estimate general X}
         \frac{\vert\nabla u\vert^{2}}{u^{2}}(x) \le C(n)\bigg(\max\bigg\{\alpha+ \frac{3}{2}(n-1)K,0\bigg\} + \beta + \Lambda^2 + \frac{1}{R^{2}}\bigg)
    \end{equation}
    where $C(n)>0$ is a constant depending only on $n$.
\end{theorem}
\begin{proof}
%STEP 1 (Setting and basic estimate)
Let $u\in C^{\infty}(B_{2R}(o))$ be a positive solution to (\ref{semilinear X-lap}). We introduce the quantities    
    \begin{equation*}
        w:=\log u, \quad  Q:=\vert\nabla w\vert^{2}.
    \end{equation*}
By the equation (\ref{semilinear X-lap}) for $u$, we have
    \begin{align*}
        \Delta_{X}w=\Delta w-g(X,\nabla w)=\frac{\Delta_{X}u}{u}-\frac{\vert\nabla u\vert^{2}}{u^{2}} =\frac{F(u)}{u}-\vert\nabla w\vert^{2},
    \end{align*}
so the function $w$ satisfies the equation
    \begin{equation*}
        \Delta_{X}w+\vert\nabla w\vert^{2}-\frac{F(u)}{u}=0.
    \end{equation*}
Moreover, by the $X$-Bochner formula, (\ref{struttura 1 F}) and the bound on $\mathrm{Ric}_{X}$, we deduce that
    \begin{align*}
         &\frac{1}{2}\Delta_{X}Q+g(\nabla w,\nabla Q) = \frac{1}{2}\Delta_{X}\vert \nabla w\vert^{2}+g(\nabla w,\nabla Q) \\
         &=\vert \mathrm{Hess}(w)\vert^{2} + g(\nabla w,\nabla\Delta_{X}w)+\mathrm{Ric}_{X}(\nabla w,\nabla w)+g(\nabla w,\nabla Q) \\
         &= \vert \mathrm{Hess}(w)\vert^{2} + g(\nabla w,\nabla(\frac{F(u)}{u}-\vert\nabla w\vert^{2}))+\mathrm{Ric}_{X}(\nabla w,\nabla w)+g(\nabla w,\nabla \vert\nabla w\vert^{2})\\
         &\ge\frac{(\Delta w)^{2}}{n} -\frac{F(u)-uF'(u)}{u^{2}}g(\nabla w,\nabla u) -(n-1)K\vert\nabla w\vert^{2} \\
         &=\frac{1}{n}(\Delta_{X} w+g(X,\nabla w))^{2} +\bigg(F'(u)-\frac{F(u)}{u}\bigg)\vert\nabla w\vert^{2}-(n-1)K\vert\nabla w\vert^{2}\\
         &\ge\frac{1}{n}(\vert\nabla w\vert^{2}-\frac{F(u)}{u}-g(X,\nabla w))^{2} - \big(\alpha+(n-1)K\big)\vert\nabla w\vert^{2},
    \end{align*}
hence
    \begin{equation}\label{conseguenza X boch}
         \Delta_{X}Q+2g(\nabla w,\nabla Q) \ge\frac{2}{n}(\vert\nabla w\vert^{2}-\frac{F(u)}{u}-g(X,\nabla w))^{2} - 2\big(\alpha+(n-1)K\big)\vert\nabla w\vert^{2}.
    \end{equation}
%STEP 2 (Cut-off function)
Now, consider a real-valued function $\psi\in C^{\infty}([0,2R])$ such that:
\begin{itemize}
    \item $0\le\psi(t)\le1$ for all $t\in[0,2R]$,
    \item $\psi_{\mid_{[0,R]}}\equiv 1$,
    \item supp($\psi$)$\subseteq [0,2R)$,
    \item $-\frac{C_{1}}{R}\psi^{\frac{1}{2}}\le\psi'\le0$ for some $C_{1}>0$,
    \item $\vert \psi''\vert\le\frac{C_{2}}{R^{2}}$ for some $C_{2}>0$.
\end{itemize}
Define the function $\Phi(x):=\psi(r(x))$, where $r(x)$ is the distance from the point $o$; then $\Phi$ is compactly supported in $B_{2R}(o)$, and, up to modifying $\Phi$ with Calabi's trick (see \cite{calabi1958extension}), we can assume that $\Phi$ is of class $C^{2}$. Finally, we set $G:=\Phi Q$. \\
%STEP 3 (Maximal technique)
Since $G$ is a non-negative function that vanishes on $\partial B_{2R}(o)$, it attains its maximum. Let $x_{0}$ be a point of maximum for $G$. Note that $x_{0}$ lies in the interior of the ball $B_{2R}(o)$, and that we can assume that $G(x_{0})$ is positive (otherwise, there is nothing to prove). Then, at $x_{0}$, we have
\begin{enumerate}
    \item $\nabla G(x_{0})=0$,
    \item $\Delta_{X}G(x_{0})=\Delta G(x_{0}) - g(X,\nabla G(x_{0})) \le 0$.
\end{enumerate}
The first condition gives, at $x_{0}$,
\begin{equation}
    \nabla Q = -Q\frac{\nabla\Phi}{\Phi}.
\end{equation}
Moreover, the second relation gives, at $x_{0}$,
    \begin{equation*}
    \begin{split}
        0&\ge\Delta_{X}G=\Delta_{X}(\Phi Q) \\
        &= \Phi\Delta_{X} Q + Q\Delta_{X}\Phi+2g(\nabla\Phi,\nabla Q)\\
        &\ge\Phi\Delta_{X} Q + Q\Delta_{X}\Phi-2\vert\nabla\Phi\vert\vert\nabla Q\vert\\
        &= \Phi\Delta_{X} Q + Q\Delta_{X}\Phi-2Q\frac{\vert\nabla\Phi\vert^{2}}{\Phi},
    \end{split}
    \end{equation*}
    hence
    \begin{equation}\label{stima 1 cutoff}
        \Phi\Delta_{X} Q \le \bigg(-\Delta_{X}\Phi+2\frac{\vert\nabla\Phi\vert^{2}}{\Phi}\bigg)Q.
    \end{equation}
    Now, we need to estimate the terms in $\Phi$ on the right-hand side of the inequality. We estimate the first term:
    \begin{equation*}
        \vert\nabla\Phi\vert^{2} = \vert \psi'(r)\nabla r\vert^{2}=(\psi'(r))^{2} \le\frac{C_{1}^{2}}{R^{2}}\psi(r) = \frac{C_{1}^{2}}{R^{2}}\Phi.
    \end{equation*}
    We estimate the second term: applying Theorem \ref{X-comparison for vector field} and using Young's inequality, we get
    \begin{equation*}
        \begin{split}
            \Delta_{X}\Phi&=\Delta\Phi-g(X,\nabla\Phi)\\
            &=\psi'(r)\Delta r+\psi''(r)-\psi'(r)g(X,\nabla r)\\
            &=\psi'(r)\Delta_{X} r+\psi''(r) \\
            &\ge-\frac{C_{1}}{R}\psi^{\frac{1}{2}}(r)\bigg(\frac{n-1}{r}+(n-1)\sqrt{K}+\Lambda\bigg)-\frac{C_{2}}{R^{2}}\\
            &\ge-\frac{C_{1}}{R}\bigg(\frac{n-1}{R}+(n-1)\sqrt{K}+\Lambda\bigg)-\frac{C_{2}}{R^{2}}\\
            &\ge-(n-1)K-\Lambda^{2} - \frac{nC_{1}^{2}}{4R^{2}}-\frac{(n-1)C_{1}+C_{2}}{R^{2}} \\
            &= -(n-1)K-\Lambda^{2} - \frac{C_{3}}{R^{2}},
        \end{split}
    \end{equation*}
    where $C_{3}:=(n-1)C_{1}+C_{2}+\frac{n}{4}C_{1}^2$, hence
    \begin{equation*}
        -\Delta_{X}\Phi\le (n-1)K+\Lambda^{2}+\frac{C_{3}}{R^{2}}.
    \end{equation*}
    Inserting these two estimates in (\ref{stima 1 cutoff}), we obtain
    \begin{equation*}
    \begin{split}
        \Phi\Delta_{X} Q &\le \bigg((n-1)K+\Lambda^{2}+\frac{C_{3}+2C_{1}^{2}}{R^{2}}\bigg)Q \\
        &=:\bigg((n-1)K+\Lambda^{2}+\frac{C_{4}}{R^{2}}\bigg)Q,
    \end{split}
    \end{equation*}
    that gives the first key estimate
    \begin{equation}\label{key estimate 1}
    \begin{split}
    \Phi\Delta_{X}Q-2g(\nabla w,\nabla\Phi)Q &\le \bigg((n-1)K+\Lambda^{2}+\frac{C_{4}}{R^{2}}\bigg)Q+2\vert\nabla w\vert\vert\nabla\Phi\vert Q \\
    &\le \bigg((n-1)K+\Lambda^{2}+\frac{C_{4}}{R^{2}}\bigg)Q+\frac{2C_{1}}{R} Q^{\frac{3}{2}}\Phi^{\frac{1}{2}}.
    \end{split}
    \end{equation}
As for the second key estimate, by exploiting (\ref{struttura 2 F}) and the bound $\vert X\vert\le\Lambda$ in (\ref{conseguenza X boch}) we get
    \begin{equation}\label{key estimate 2}
    \begin{split}
        &\Phi\Delta_{X}Q-2g(\nabla w,\nabla\Phi)Q=(\Delta_{X}Q+2g(\nabla w,\nabla Q))\Phi\\
        &\ge \bigg[\frac{2}{n}(\vert\nabla w\vert^{2}-\frac{F(u)}{u}-g(X,\nabla w))^{2} - 2\big(\alpha+(n-1)K\big)\vert\nabla w\vert^{2}\bigg]\Phi \\
        &= \bigg[\frac{2}{n}\bigg(\vert\nabla w\vert^{4}+\bigg(\frac{F(u)}{u}\bigg)^{2}+g(X,\nabla w)^{2}-2\frac{F(u)}{u}\vert \nabla w\vert^{2}-2\vert\nabla w\vert^{2}g(X,\nabla w) \\
        &+2\frac{F(u)}{u}g(X,\nabla w)\bigg) - 2\big(\alpha+(n-1)K\big)\vert\nabla w\vert^{2}\bigg]\Phi \\
        &\ge \bigg[\bigg(\frac{2}{n}\vert\nabla w\vert^{4}-\frac{4}{n}\beta\vert \nabla w\vert^{2}-\frac{4\Lambda}{n}\vert\nabla w\vert^{3}-\frac{4\Lambda}{n}\beta\vert\nabla w\vert\bigg) \\
        & - 2\big(\alpha+(n-1)K\big)\vert\nabla w\vert^{2}\bigg]\Phi \\
        &\ge \bigg[\frac{2}{n}Q^{2}-\frac{4\Lambda}{n}Q^{\frac{3}{2}}- 2\bigg(\alpha +\frac{2}{n}\beta +(n-1)K\bigg)Q-\frac{4\Lambda}{n}\beta Q^{\frac{1}{2}}\bigg]\Phi.
    \end{split}
    \end{equation}
    Combining (\ref{key estimate 1}) and (\ref{key estimate 2}), both multiplied by $\Phi$, and using the fact that $\Phi\le1$, we get
    \begin{equation*}
    \begin{split}
        \frac{2}{n}G^{2}-\frac{4\Lambda}{n}G^{\frac{3}{2}}- 2\bigg(\alpha +\frac{2}{n}\beta +(n-1)K\bigg)G-\frac{4\Lambda}{n}\beta G^{\frac{1}{2}} \le \bigg((n-1)K+\Lambda^{2}+\frac{C_{4}}{R^{2}}\bigg)G+\frac{2C_{1}}{R} G^{\frac{3}{2}},
        \end{split}
    \end{equation*}
    which can be written as
    \begin{equation}\label{equazione intermedia}
    \begin{split}
        \frac{2}{n}G^{2}-2\bigg(\frac{2\Lambda}{n}+\frac{C_{1}}{R}\bigg)G^{\frac{3}{2}}- 2\bigg(\alpha +\frac{2}{n}\beta +\frac{3}{2}(n-1)K+\Lambda^{2}+\frac{C_{4}}{R^{2}}\bigg)G-\frac{4\Lambda}{n}\beta G^{\frac{1}{2}} \le 0.
        \end{split}
    \end{equation}
    Now, we use Young inequality on the term in $G^{\frac{3}{2}}$ to find
    \begin{equation*}
    \begin{split}
        -\bigg(\frac{4\Lambda}{n}+\frac{2C_{1}}{R}\bigg)G^{\frac{3}{2}}&\ge-\frac{1}{n}G^{2}-n\bigg(\frac{2\Lambda}{n}+\frac{C_{1}}{R}\bigg)^{2}G \\
        & \ge-\frac{1}{n}G^{2}-\bigg(\frac{8\Lambda^{2}}{n}+\frac{2n^{2}C^{2}_{1}}{R^{2}}\bigg)G,
    \end{split}
    \end{equation*}
    so that (\ref{equazione intermedia}) becomes
    \begin{equation*}\begin{split}
        \frac{1}{n}G^{2}- 2\bigg(\tilde{\alpha} +\frac{2}{n}\beta +\Lambda^{2}+\frac{4\Lambda^{2}}{n}+\frac{C_{5}}{R^{2}}\bigg)G-\frac{4\Lambda}{n}\beta G^{\frac{1}{2}} \le 0,
        \end{split}
    \end{equation*}
    where  
    \begin{equation*}
        C_{5}:=C_{4}+2n^{2}C_{1}^{2}, \quad \tilde{\alpha}:=\max\bigg\{\alpha+\frac{3}{2}(n-1)K,0\bigg\}.
    \end{equation*}
    Since $G(x_{0})>0$, we can divide by $\frac{1}{n}G^{\frac{1}{2}}(x_{0})$, and we obtain
    \begin{equation}\label{pre terzo grado}
        G^{\frac{3}{2}}- \bigg(2n\tilde{\alpha} +4\beta +8\Lambda^{2}+2n\Lambda^{2}+\frac{2nC_{5}}{R^{2}}\bigg)G^{\frac{1}{2}}-4\Lambda\beta \le 0.
    \end{equation}
    Set
    \begin{equation*}
        \begin{split}
            t&:=G^{\frac{1}{2}}(x_{0}), \\
            A&:=2n\tilde{\alpha} +2n\Lambda^{2}+\frac{2nC_{5}}{R^{2}} \ge0,
        \end{split}
    \end{equation*}
    then (\ref{pre terzo grado}) becomes
    \begin{equation*}
        \begin{split}
        t^{3} - (4\beta +8\Lambda^{2} + A)t-4\Lambda\beta \le 0.
        \end{split}
    \end{equation*}
    We need to find the roots of the equation $t^{3}+pt+q=0$, where 
    \begin{equation*}
            p:=-(4\beta +8\Lambda^{2} + A), \quad q:=-4\Lambda\beta.
    \end{equation*}
    If $\beta=0$ or $\Lambda=0$, then the polynomial is $t^{3}+pt$, and the roots are
    \begin{equation*}
        r_{1}=-\sqrt{-p}, \quad r_{2}=0, \quad r_{3} = \sqrt{-p}. 
    \end{equation*}
    Since $t=G^{\frac{1}{2}}(x_{0})>0$, this implies that $t^{3}+pt\le0$ if and only if
    \begin{equation*}
        t\le \sqrt{-p}.
    \end{equation*}
    If $\beta>0$, note that the discriminant of the polynomial, which is  
    \begin{equation*}
        -(4p^{3}+27q^{2}),
    \end{equation*}
    is non-negative: indeed, 
    \begin{equation*}
    \begin{split}
        -(4p^{3}+27q^{2}) &= 4(4\beta +8\Lambda^{2} + A)^{3}-27(-4\Lambda\beta)^{2} \\
        &=4\sum_{i=0}^{3}\binom{3}{i}(4\beta)^{i}(8\Lambda^{2}+A)^{3-i}-432\beta^{2}\Lambda^{2} \\
        &=4\sum_{i=0}^{3}\sum_{j=0}^{3-i}\binom{3}{i}\binom{3-i}{j}(4\beta)^{i}(8\Lambda^{2})^{j}A^{3-i-j}-432\beta^{2}\Lambda^{2} \\
        &\ge12(4\beta)^{2}(8\Lambda^{2})-432\beta^{2}\Lambda^{2} \\
        &= 1104\beta^{2}\Lambda^{2} >0.
    \end{split}
    \end{equation*}
    This means that there are exactly three real roots. In particular, there is exactly one real root on $[0,+\infty)$, and this can be seen by looking at the sign of the first derivative of $t^{3}+pt+q$. Let $r_{1},r_{2},r_{3}$ be the roots of $t^{3}+pt+q$, with $r_{1},r_{2}<0$ and $r_{3}>0$; since $G^{\frac{1}{2}}(x_{0})>0$, we have $t^{3}+pt+q=(t-r_{1})(t-r_{2})(t-r_{3})\le0$ if and only if $t\le r_{3}$. Moreover, the expression of the roots in this case is
    \begin{equation*}
        r_{\ell}=-2\sqrt{-\frac{p}{3}}\sin\bigg(\frac{1}{3}\arcsin\bigg(\frac{3\sqrt{3}q}{2p\sqrt{-p}}\bigg)+\frac{2\ell\pi}{3}\bigg), \quad \ell=-1,0,1,
    \end{equation*}
    hence
    \begin{equation*}
        t\le\frac{2}{\sqrt{3}}\sqrt{-p}.
    \end{equation*}
    Summing up, for $\beta \ge 0$ we have that
    \begin{equation*}
        t\le\frac{2}{\sqrt{3}}\sqrt{-p},
    \end{equation*}
    hence
    \begin{equation*}
    \begin{split}
        G(x_{0}) &= t^{2} \le -\frac{4}{3}p \\
        &=\frac{4}{3}\bigg(2n\tilde{\alpha} +4\beta +(8+2n)\Lambda^{2}+\frac{2nC_{5}}{R^{2}}\bigg) \\
        & \le C(n)\bigg(\tilde{\alpha} + \beta +  \Lambda^{2} + \frac{1}{R^{2}}\bigg).
    \end{split}
    \end{equation*}
    Finally, for every $x\in B_{R}(o)$ we have
    \begin{equation*}
        \frac{\vert\nabla u\vert^{2}}{u^{2}}(x) \le G(x_{0})  \le C(n)\bigg(\tilde{\alpha} + \beta +  \Lambda^{2} + \frac{1}{R^{2}}\bigg).
    \end{equation*}
\end{proof}
\begin{remark}
    The gradient estimate in Theorem \ref{general gradient estimate for X-Lap} can be applied to the case of $X$-harmonic functions, which corresponds to $\alpha=\beta=0$. Other examples of $F$ that satisfy the structural conditions of the theorem are $F(t)=\lambda t$, with $\lambda\in\mathbb{R}$, $F(t)=\mathrm{sin}(t)$ or $F(t)=\frac{at}{(t+b)^{\sigma}}$, where $a\in\mathbb{R}$, $b>0$ and $\sigma\ge0$.
\end{remark}
\begin{remark}
    In the case of $X$-harmonic functions, the result in Theorem \ref{general gradient estimate for X-Lap} extends to the non-gradient case the result in \cite{Li-ChenGradientActa}, and removes the assumption on the Ricci curvature, requiring only a lower bound on the Bakry–Émery-Ricci curvature.
\end{remark}
\begin{remark}
    During the preparation of this work, an analogous estimate for $X$-harmonic functions has appeared in a paper on arXiv (see \cite{dong2025gradientestimatespvharmonicfunctions}), where the authors established, with different techniques, the validity of similar gradient estimate for the drifted $p$-Laplace equation.
\end{remark}
Integrating the gradient estimate along geodesics, we obtain the following local Harnack inequality for positive solutions to (\ref{semilinear X-lap}).
\begin{corollary}\label{X-Harnack}
Let $(M,g)$ be a complete Riemannian manifold of dimension $n$, and let $X$ be a smooth vector field on $M$. Let $o\in M$ be a fixed point, let $R>0$, and assume that $\mathrm{Ric}_{X}\ge-(n-1)K$ and $\vert X\vert\le \Lambda$ on the ball $B_{2R}(o)$, for some constants $K,\Lambda\ge0$. Let $u\in C^{\infty}(B_{2R}(o))$ be a positive solution to
    \begin{equation*}
        \Delta_{X}u=F(u) \quad \text{on $B_{2R}(o)$},
    \end{equation*}
    where $F\in C^{\infty}([0,+\infty))$ satisfies the structural conditions (\ref{struttura 1 F}) and (\ref{struttura 2 F})    
    for some $\alpha\in\mathbb{R}$ and $\beta \ge 0$. Then, for every $x,y\in B_{R/2}(o)$ we have the estimate
    \begin{equation}\label{formula of X-Harnack}
        u(y)\le C_{2}\mathrm{exp}\big(C_{1}(\gamma +\Lambda)R\big)u(x),
    \end{equation}
    where $C_{1},C_{2}>0$ are constants depending only on $n$ and 
    \begin{equation*}
        \gamma:=\sqrt{\max\bigg\{\alpha+ \frac{3}{2}(n-1)K,0\bigg\}+ \beta}.
    \end{equation*}
\end{corollary}
\begin{remark}
    The gradient estimates in Theorem \ref{general gradient estimate for X-Lap} and the local Harnack inequality have been stated assuming some structural conditions on the semilinearity $F$. However, due to their local nature, the results can also be stated in the case of a general $F$, up to replacing the constants $\alpha$ and $\beta$ with, respectively,
    \begin{equation*}
        \sup_{B_{2R}(o)}\bigg\{\frac{F(u)}{u}-F'(u)\bigg\}
    \end{equation*}
    and 
    \begin{equation*}
        \sup_{B_{2R}(o)}\frac{\vert F(u)\vert}{u}.
    \end{equation*}
\end{remark}

In contrast to the case of harmonic functions, the previous results do not yield global estimates due to the presence of the term $\sup_{B_{2R}(o)}\vert X\vert$, even when the Bakry–Émery-Ricci curvature is globally bounded from below, as illustrated by the following example in the $X$-harmonic case. 

\begin{example}
We start by considering the 1-dimensional case: on $\mathbb{R}$ with the standard flat metric, fix a smooth function $b\in C^{\infty}(\mathbb{R})$, and consider the equation associated to the drifted Laplacian with coefficient $b$, which is
\begin{equation}\label{ODE}
    u''(x)-b(x)u'(x)=0.
\end{equation}
By setting $v(x):=u'(x)$, one can solve this ODE, and finds that the solutions are of the form
\begin{equation}\label{solution of ODE}
u(x)=c_{1}+c_{2}\int_{0}^{x}\textnormal{exp}\bigg(\int_{0}^{t}b(s)\ \textnormal{d}s\bigg)\textnormal{d}t,
\end{equation}
where $c_{1}$ and $c_{2}$ are real numbers. Note that De L'Hopital theorem implies
\begin{equation*}
    \lim_{x\rightarrow+\infty}\frac{u'(x)}{u(x)}=\lim_{x\rightarrow+\infty}\frac{u''(x)}{u'(x)}=\lim_{x\rightarrow+\infty}b(x).
\end{equation*}
Thus, the behavior of $b(x)$ at infinity is related to that of $(\log(u))'$. In particular, this shows that a global gradient estimate cannot hold if the drift term is unbounded at infinity. 
\\ 
Now, consider $\mathbb{R}^{n}$ with the standard flat metric and coordinates $x_{1},\dots,x_{n}$, and define the smooth vector field $X:=b(x_{1})\frac{\partial}{\partial x_{1}}$. A quick computation shows that the associated Bakry-\'Emery-Ricci curvature is 
\begin{equation*}
    \mathrm{Ric}_{X}=b'(x_{1})\mathrm{d}x_{1}\otimes\mathrm{d}x_{1},
\end{equation*}
which is non-negative if and only if $b'(t)\ge0$. Moreover, the function $u$ defined in (\ref{solution of ODE}) can be extended to $\mathbb{R}^{n}$ to produce an $X$-harmonic function, by defining
\begin{equation*}
    u(x_{1},\dots,x_{n})\equiv u(x_{1})=c_{1}+c_{2}\int_{0}^{x_{1}}\textnormal{exp}\bigg(\int_{0}^{t}b(s)\ \textnormal{d}s\bigg)\textnormal{d}t.
\end{equation*}
To construct an explicit counterexample to the validity of (\ref{global gradient estimate for X-Lap}), we pick the function $b\in C^{\infty}(\mathbb{R})$ such that:
\begin{enumerate}
    \item $b(x)=1$ for $x\le0$;
    \item $b(x)=h(x)$ for $x\ge x_{0}$, where $x_{0}>0$ and $h$ is a smooth, non-decreasing function such that $h\ge1$;
    \item $b'(x)\ge0$.
\end{enumerate}
If we choose $b$ with all these properties, then the function $u$ is bounded as $x$ tends to $-\infty$. Indeed, we have
\begin{equation*}
\begin{split}
    \int_{-\infty}^{0}\textnormal{exp}\bigg(\int_{0}^{t}b(s)\ \textnormal{d}s\bigg)\textnormal{d}t =\int_{-\infty}^{0}e^{t}\ \textnormal{d}t = 1,
\end{split}
\end{equation*}
hence
\begin{equation*}
\begin{split}
    \lim_{x\rightarrow-\infty}u(x)=c_{1}+c_{2}\int_{0}^{-\infty}\textnormal{exp}\bigg(\int_{0}^{t}b(s)\ \textnormal{d}s\bigg)\textnormal{d}t=c_{1}-c_{2}.
\end{split}
\end{equation*}
If we choose $c_{1}=c_{2}=1$, then $u$ is a non-constant, positive $X$-harmonic function that does not satisfy (\ref{global gradient estimate for X-Lap}) for any choice of $h$ as above; in particular, we have counterexamples to the validity of the global gradient estimates when the growth of the norm of $X$ is sub-linear (choose $h(x)=(1+x^{2})^{\frac{\delta}{2}}$, $\delta\in(0,1)$) or logarithmic (choose $h(x)=\log(1+x^{2})$).
\end{example}

\noindent However, when $\mathrm{Ric}_{X}$ is bounded from below and $X$ is a bounded vector field, the gradient estimate gives global information on the quantity $\vert\nabla \log u\vert$, where $u$ is a positive $X$-harmonic function. We have the following

\begin{corollary}\label{global gradient estimate}
    Let $(M,g)$ be a complete Riemannian manifold of dimension $n$, and let $X$ be a smooth vector field on $M$. Assume that $\mathrm{Ric}_{X}\ge-(n-1)K$ and $\vert X\vert\le\Lambda$ on $M$, for some constants $K,\Lambda\ge0$. Let $u\in C^{\infty}(M)$ be a positive solution to
    \begin{equation}
        \Delta_{X}u=F(u) \quad \text{on $M$},
    \end{equation}
    where $F\in C^{\infty}([0,+\infty))$ satisfies the structural conditions (\ref{struttura 1 F}) and (\ref{struttura 2 F})    
    for some $\alpha\in\mathbb{R}$ and $\beta \ge 0$. Then, for every $x\in M$ we have the estimate
    \begin{equation}\label{global gradient estimate for X-Lap}
        \frac{\vert \nabla u\vert^{2}}{u^{2}}(x)\le C(n)\bigg(\max\bigg\{\alpha+ \frac{3}{2}(n-1)K,0\bigg\}+ \beta+\Lambda^{2}\bigg),
    \end{equation}
    where $C(n)>0$ is a constant depending only on $n$.
\end{corollary}

\section{Liouville theorem for positive $X$-harmonic functions}
In this section, we discuss the Liouville property for positive, globally defined $X$-harmonic functions. 
\\
Let $(M,g)$ be a complete, non-compact Riemannian manifold of dimension $n$, and consider a smooth vector field $X$ on $M$. We say that $(M,g,X)$ satisfies the Liouville property for the $X$-Laplacian if every positive $X$-harmonic function on $M$ is constant. 

Our aim is to establish conditions on $X$ that guarantee the validity of the Liouville property when the Bakry-\'Emery-Ricci curvature is non-negative. As a first step, we note that the proof of Proposition 3.1 in \cite{munteanu2011smooth} can be adapted to get the following version of the Liouville theorem in the non-gradient case.
\begin{theorem}\label{X-Munteanu-Wang}
    Let $(M,g)$ be a complete, non-compact Riemannian manifold of dimension $n$, let $X$ be a smooth vector field on $M$, and fix a reference point $o\in M$. Assume that $\mathrm{Ric}_{X}\ge 0$. If $u\in C^{\infty}(M)$ is a positive $X$-harmonic function, then
    \begin{equation*}
        \sup_{M}\vert\nabla\log u\vert^{2}\le C(n,o)(\Omega(u)+\Omega(u)^{2}),
    \end{equation*}
    where 
    \begin{equation*}
        \Omega(u):=\limsup_{R\rightarrow +\infty}\bigg[\frac{1}{R}\sup_{B_{R}(o)}\log(u+1)\bigg]
    \end{equation*}
    and $C(n,o)>0$ is a constant depending only on $n$ and $o$. In particular, every positive $X$-harmonic function of sub-exponential growth on $M$ is constant.
\end{theorem}
\begin{proof}
    The proof is obtained by repeating the argument in \cite{munteanu2011smooth}, applying the maximum principle method to the function $u^{\varepsilon}$, $\varepsilon>0$. However, in this case one must apply the comparison theorem in \cite{catino2016analytic} with $F(t)\equiv 0$, which gives the estimate
    \begin{equation*}
        \Delta_{X} r\le C
    \end{equation*}
    for some constant $C>0$ depending on the point $o$. 
\end{proof}
\begin{remark}
    The growth condition of the solution is sharp: indeed, if we consider $\mathbb{R}^{n}$ with the standard metric and the vector field $X:=\frac{\partial}{\partial x_{1}}$, then $\mathrm{Ric}_{X}\ge0$ and the function $u(x):=e^{x_{1}}$ is a positive $X$-harmonic function of exponential growth (see \cite{brighton2013liouville}).   
\end{remark}
\begin{remark}
    It is interesting to observe that if $\mathrm{Ric}_{X}\ge0$ and $X$ is bounded, then any positive $X$-harmonic function on $M$ is either constant or has exponential growth. This follows immediately from the global gradient estimate and Theorem \ref{X-Munteanu-Wang}.
\end{remark}
\begin{remark}
    The comparison theorem in \cite{catino2016analytic}, which is used in the proof of Theorem \ref{X-Munteanu-Wang}, is a powerful tool, as it does not require any growth assumption on the vector field $X$. However, the constants depend strongly on the center of the ball under consideration, and this dependence would persist if we chose to use this result instead of Theorem \ref{X-comparison for vector field} in the proof of the gradient estimate.
\end{remark}

We now show that if the norm of the vector field decays at infinity, the result in Theorem \ref{X-Munteanu-Wang} holds for general positive $X$-harmonic functions, recovering the full Liouville theorem. We have the following
\begin{theorem}\label{Liouville}
    Let $(M,g)$ be a complete, non-compact Riemannian manifold of dimension $n$, and let $X$ be a smooth vector field on $M$. Suppose that $\mathrm{Ric}_{X}\ge 0$ and that there exists a point $o\in M$ such that
    \begin{equation}\label{decay of Lambda(r)}
        \vert X\vert\le \Lambda(r(x)) \quad\quad \forall x\in M,
    \end{equation}
    where $r(x)$ is the distance from $o$ and $\Lambda: [0,+\infty)\longrightarrow [0,+\infty)$ a continuous function satisfying
    \begin{equation*}
        \lim_{t\rightarrow+\infty}\Lambda(t) = 0.
    \end{equation*}
    Then, every positive, globally defined $X$-harmonic function on $M$ is constant.
\end{theorem}
\begin{proof}
    Let $u\in C^{\infty}(M)$ be a positive $X$-harmonic function. As in the proof of Theorem \ref{general gradient estimate for X-Lap}, we introduce the quantities
    \begin{equation*}
        w:=\log u, \quad Q:=\vert\nabla w\vert^{2}.
    \end{equation*}
    Since $\Lambda(t)$ tends to zero at infinity, we deduce that for every $k\in\mathbb{N}$ there exists $R_{k}>0$ such that
    \begin{equation}\label{lambda bound on k}
        \vert \Lambda(t)\vert\le\frac{1}{k+1} \quad \forall t\ge R_{k},
    \end{equation}
    and the sequence $\{R_{k}\}_{k\ge1}$ can be chosen to be increasing and divergent.  Having fixed $k\in\mathbb{N}$, we can combine the estimates \eqref{lambda bound on k} and $\vert X\vert\le\Lambda(r)$ to find 
    \begin{equation}\label{X-bound on exterior of a ball}
        \vert X\vert \le \frac{1}{k+1} \quad \forall x \in M\setminus B_{R_{k}}(o). 
    \end{equation}
    In particular, since $B_{R_{k}}(x)\subseteq M\setminus B_{R_{k}}(o)$ for every $x\in M\setminus B_{2R_{k}}(o)$, we have that (\ref{X-bound on exterior of a ball}) holds on the ball $B_{R_{k}}(x)$. By applying estimate (\ref{gradient estimate general X}), we obtain
    \begin{equation*}
        Q(x)\le C(n)\bigg(\frac{1}{k^{2}}+\frac{4}{R_{k}^{2}}\bigg).
    \end{equation*}
    This holds for every $x\in M\setminus B_{2R_{k}}(o)$, and the right-hand side of the inequality goes to zero as $k$ tends to infinity: it follows that
    \begin{equation*}
        \lim_{x\rightarrow\infty}Q(x)=0.
    \end{equation*}
    Since $Q$ is non-negative and decays to zero at infinity, we deduce the existence of a maximum point for $Q$ in $M$. Moreover, arguing as in Theorem \ref{general gradient estimate for X-Lap}, the $X$-Bochner formula and the non-negativity of the Bakry–Émery-Ricci curvature imply that
    \begin{equation*}
        \Delta_{X}Q+2g(\nabla w,\nabla Q) \ge 0.
    \end{equation*}
    This means that $Q$ is a $Y$-subharmonic function, where $Y$ is the vector field on $M$ given by 
    \begin{equation*}
        Y:=X-2\nabla w,
    \end{equation*}
    which is globally bounded by the global gradient estimate \eqref{global gradient estimate for X-Lap}. Since $Q$ is $Y$-subharmonic, by the maximum principle and the fact that it has an interior maximum we deduce that $Q$ is constant. Moreover, since $Q$ goes to zero at infinity we have $Q\equiv0$, so $\nabla u\equiv0$, and $u$ is constant.
\end{proof}
\begin{remark}
    We observe that there are no assumptions on the decay at infinity of the norm of $X$. Therefore, the function $\Lambda(t)$ need not be decreasing or satisfy any prescribed rate of decay at infinity.
\end{remark}

\begin{remark}\label{optimality_Liou_positive_decaying}
	We stress that the limit in condition \eqref{decay of Lambda(r)} of Theorem \ref{Liouville} cannot be replaced by the liminf.
	Indeed, on $\mathbb{R}^{n}$ with the standard flat metric and coordinates $x_{1},\dots,x_n$, let us consider the smooth vector field
	\begin{equation*}
		X=\varphi(x_{n})\frac{\partial}{\partial x_{n}},
	\end{equation*}
	where
	\begin{equation*}
		\varphi(t)=\frac{6}{D(t)}, \quad\quad    D(t)=1+\sqrt{1+t^{2}}-t.
	\end{equation*}
	Note that the functions $D(t)$ and $\varphi(t)$ satisfy the following properties:
	\begin{itemize}
		\item $D(t)\ge1>0$ and $0<\varphi(t)\le 6$ for every $t\in\mathbb{R}$,
		\item $D'(t)=\frac{t}{\sqrt{1+t^{2}}}-1 < 0$, hence $\varphi'(t)=-\frac{6D'(t)}{D^{2}(t)}>0$,
		\item $\lim_{t\rightarrow-\infty}\varphi(t)=0$, and $\lim_{t\rightarrow+\infty}\varphi(t)=6$,
		\item $\varphi(t)\ge\frac{3}{1+\vert t\vert}$ for all $t\le0$.
	\end{itemize}
	From the fact that $\varphi'(t)\ge0$ it immediately follows that $\mathrm{Ric}_{X}\ge0$. Moreover, $\vert X\vert\le 6$ and
	\begin{align*}
		\lim_{x_{n}\rightarrow-\infty}\vert X\vert(0,\dots,0,x_{n})&=0, \\
		\lim_{x_{n}\rightarrow+\infty}\vert X\vert(0,\dots,0,x_{n})&=6,
	\end{align*}
	so that $\liminf_{x\rightarrow\infty}\vert X\vert=0$ but the limit does not exist.	Let us define the smooth function $u(x):=U(x_{n})$, where
	\begin{equation*}
		U(t):=\int_{-\infty}^{t}e^{\psi(s)}\ \mathrm{d}s, \quad\quad \psi(t)=\int_{0}^{t}\varphi(s)\ \mathrm{d}s.
	\end{equation*}
	Note that $u$ is well defined and finite everywhere, since the estimate
	\begin{equation*}
		\psi(t)=-\int_{t}^{0}\varphi(s)\ \mathrm{d}s\le-\int_{t}^{0}\frac{3}{1+\vert s\vert}\ \mathrm{d}s=-3\log(1+\vert t\vert)
	\end{equation*}
	implies
	\begin{equation*}
		0\le \int_{-\infty}^{0}e^{\psi(s)}\ \mathrm{d}s\le\int_{-\infty}^{0}\frac{1}{(1+\vert s\vert)^{3}}\ \mathrm{d}s<+\infty.
	\end{equation*}
	Clearly, $u$ is a positive and non-constant function. Moreover, $u$ is $X$-harmonic: indeed,
	\begin{equation*}
		\Delta_{X} u = \frac{\partial^{2}U}{\partial x_{n}^{2}}-\varphi(x_{n})\frac{\partial U}{\partial x_{n}} = \varphi(x_{n})e^{\psi(x_{n})}-\varphi(x_{n})e^{\psi(x_{n})} = 0.
	\end{equation*}
	Hence, the strong Liouville theorem for $X$-harmonic functions does not hold. For this reason, the decay condition in Theorem \ref{Liouville} is the optimal condition to get an analogue, in the non-gradient setting, of the Liouville theorem for positive $f$-harmonic functions by Munteanu and Wang.
\end{remark}

\begin{remark}
    An example of Riemannian manifold $(M,g)$ with a vector field $X$ satisfying the hypothesis of Theorem \ref{Liouville} can be constructed by taking any manifold $(M,g)$ with positive Ricci curvature and a compactly supported vector field $X$ such that the norm of $\mathcal{L}_{X}g$ is small enough on the support. 
    \\ Other examples can be constructed on $\mathbb{R}^{2}$ endowed with the complete metric 
    \begin{equation*}
        g=\begin{pmatrix}
            1+4x^{2} & 4xy\\
            4xy & 1+4y^{2}
        \end{pmatrix}
    \end{equation*}
    by considering:
    \begin{enumerate}
        \item the gradient of the function $\Phi(x,y):=\frac{1}{8}\log(1+4x^{2}+4y^{2})$;
        \item a vector field $X$ obtained by perturbing a gradient vector field with a suitable non-gradient perturbation;
        \item the vector field $X=fY$, where $Y=(-y,x)$ is the Killing field generating the rotations on the plane and $f$ is a decreasing function that controls the growth of the Lie derivative of $Y$.
    \end{enumerate} 
    For more details and the explicit computations, we refer to the Appendix A.
\end{remark}

\noindent In summary, as concerns the gradient estimate (\ref{global gradient estimate for X-Lap}) and the Liouville theorem for positive $X$-harmonic functions on manifolds with $\mathrm{Ric}_{X}\ge0$, the situation can be summarized as follows:
\begin{itemize}
    \item if $X$ is bounded and its norm decays to zero at infinity, then the manifold supports the gradient estimate (\ref{global gradient estimate for X-Lap}) and the Liouville property;
    \item if $X$ is bounded but its norm does not decay, then the gradient estimate \eqref{global gradient estimate for X-Lap} holds on the manifold, and in general, the Liouville property can be guaranteed only for positive solutions with sub-exponential growth;
    \item if $X$ is unbounded, the Liouville property can be guaranteed only for positive solutions with sub-exponential growth, and in general the manifold does not support the gradient estimate (\ref{global gradient estimate for X-Lap}).
\end{itemize}

\section{Liouville theorem for integrable $X$-subharmonic functions}
As we saw in the introduction, a second class of Liouville-type results involves non-negative subharmonic functions satisfying suitable $L^{p}$-integrability conditions. In this direction, the situation regarding the weighted Laplace operator on metric measure spaces is well understood (see \cite{yau1976some,PigolaRS-Vanishing}, and \cite{PVV} for the distributional setting). However, in the non-gradient setting, where the operator is no longer symmetric with respect to the (possibly weighted) Riemannian volume measure, a version of the classical Yau theorem for integrable $X$-subharmonic functions fails in general, as shown by the following example.
\begin{example}\label{failure_LpLiou_notInv}
    Let $\mathbb{T}^{n-1}$ be the $(n-1)$-dimensional flat torus, and consider the product manifold $M:=\mathbb{R}\times \mathbb{T}^{n-1}$ endowed with the warped metric $g=\mathrm{d}t^{2} + e^{2t}g_{\mathbb{T}^{n-1}}$. Pick a smooth vector field of the form $X:=\alpha\partial_{t}$, $\alpha>n-1$, and set $\beta:=\alpha-(n-1)$. If we introduce the smooth function 
    \begin{equation*}
        u(t):=\int_{t}^{+\infty}\frac{e^{\beta s}}{1+e^{(\beta+n-1)s}}\ \mathrm{d}s,
    \end{equation*}
    a straightforward computation shows that $u$ is a non-negative $X$-subharmonic function. Moreover, from the asymptotics of $u(t)$ at $\pm\infty$ it is not hard to see that $u\in L^{p}(\mathbb{R},e^{(n-1)t}\mathrm{d}t)$ for every $p\in(1,+\infty)$, hence $u\in L^{p}(M,\mathrm{dv})$ for every $p\in(1,+\infty)$. 
\end{example}
\noindent For this reason, to recover what happens in the symmetric case one is lead to consider $L^{p}$-Liouville theorems with respect to measures which are suitably related with the drifted Laplacian. Using a probabilistic viewpoint, Pessoa-Pigola-Setti identified a class of special invariant measures (see \cite{PessoaPigolaSetti2026volumegrowthparabolicitynonsymmetric}) as the right framework to extend to the non-symmetric setting classical results in the symmetric case.

\subsection{Invariant measures}
Let $(M,g)$ be a Riemannian manifold endowed with a smooth vector field $X$, and denote by $\Delta_{X}^{*,\mathrm{dv}}$ the adjoint of $\Delta_{X}$ with respect to the scalar product on the Hilbert space $L^2(M,\mathrm{dv})$. Namely, the two operators satisfy the relation
\begin{equation*}
    \int_{M}\varphi\Delta_{X}\psi\ \mathrm{dv} = \int_{M}\psi\Delta_{X}^{*,\mathrm{dv}}\varphi\ \mathrm{dv} \quad\quad \varphi,\psi\in C^{\infty}_{c}(M).
\end{equation*}
Following \cite{PessoaPigolaSetti2026volumegrowthparabolicitynonsymmetric} (see also \cite{bogachev2022fokker} for the Euclidean case), we introduce a class of measures related to the operator $\Delta_X$.
\begin{definition}
    A positive Radon measure $\mu$ on $(M,g)$ is said to be $\mathbf{invariant}$ for the operator $\Delta_{X}$ if it satisfies
    \begin{equation*}
        \Delta_{X}^{*,\mathrm{dv}}\mu = 0
    \end{equation*}
in the distributional sense.  
\end{definition}

\noindent Namely, a positive Radon measure $\mu$ is invariant if the relation
    \begin{equation*}
        \int_{M}\Delta_{X}\varphi \ \mathrm{d}\mu=0
    \end{equation*}
holds for every $\varphi\in C^{\infty}_{c}(M)$.

\begin{example}[Riemannian density]\label{Exmp_Riemm_density}
    If $(M,g)$ is any complete, non compact Riemannian manifold and $X$ is a smooth vector field, then the Riemannian measure is invariant if and only if div$X=0$. Indeed, if $\varphi\in C^{\infty}_{c}(M)$ then
    \begin{equation*}
        \int_{M}\Delta_{X}\varphi \ \mathrm{dv} = -\int_{M}g(X,\nabla\varphi) \ \mathrm{dv}  = \int_{M}\varphi \mathrm{div}X \ \mathrm{dv},
    \end{equation*}
    and the right-hand side is zero for every $\varphi\in C^{\infty}_{c}(M)$ if and only if div$X\equiv0$.
\end{example}

\noindent When d$\mu=u \mathrm{dv}$, $0<u\in C^{\infty}(M)$, then the measure is invariant for $\Delta_{X}$ if $u$ is a solution of the Fokker-Planck-Kolmogorov equation
\begin{equation*}
    \begin{cases}
        L_{X} u = 0, \\
        u>0,
    \end{cases}
\end{equation*}
where the operator $L_{X}$ is defined as
\begin{equation*}
    L_{X}u:=\mathrm{div}(\nabla u +uX).
\end{equation*}
In the next lemma, we collect some results concerning the existence, uniqueness and regularity of invariant measures for the operator $\Delta_{X}$.

\begin{lemma}[\cite{bogachev2022fokker},\cite{PessoaPigolaSetti2026volumegrowthparabolicitynonsymmetric}]
    Let $(M,g)$ be a complete, non compact Riemannian manifold and let $X$ be a smooth vector field on $M$. 
    \begin{enumerate}
        \item If $\mu$ is a positive Radon measure which is invariant for $\Delta_{X}$, then $\mathrm{d}\mu=u\mathrm{dv}$ for some positive function $u\in C^{\infty}(M)$.
        \item There exists at least one invariant measure for $\Delta_{X}$.
        \item There exists exactly one invariant measure (up to positive multiplicative constants) if and only if every positive solution $u\in C^{\infty}(M)$ to the equation
        \begin{equation*}
            \mathrm{div}(\nabla u + uX)=0
        \end{equation*}
        is constant.
    \end{enumerate}
\end{lemma}
\begin{remark}
    It is worth noticing that it is not uncommon for a complete, non-compact Riemannian manifold $(M,g)$ to support many invariant measures for $\Delta_{X}$. For instance, if the cone of positive harmonic function is non-trivial (as in the case of Cartan-Hadamard manifolds with bounded curvature pinched away from zero), then there are multiple "independent" invariant measure for the Laplacian $\Delta$ (see Example 2.18 in \cite{PessoaPigolaSetti2026volumegrowthparabolicitynonsymmetric}).
\end{remark}

\subsection{$X$-Laplacian cutoff functions}
Having introduced the framework of invariant measures, the second ingredient needed to establish an $L^{p}$-Liouville theorem for non-negative $X$-subharmonic function is a special class of cutoff functions, which can be viewed as the analogue, in the non-gradient setting, of the Laplacian cutoff functions. In what follows, when we fix a reference point $o\in M$ we will denote by $r(x)$ the distance function from $o$.

\begin{theorem}[$X$-Laplacian cutoff functions]\label{existence of X-Lap cutoff}
    Let $(M,g)$ be a complete Riemannian manifold, fix a point $o\in M$ and denote by $B_{R}$ the ball of center $o$ and radius $R$. Let $X$ be a smooth vector field such that $\vert X\vert\ \le\Lambda$ and $\mathrm{Ric}_{X}\ge-(n-1)K$ for some $\Lambda, K\ge0$. For every $R\ge2$ there exist $\varphi_{R}\in C^{\infty}_{c}(M)$ and constants $C>0$ and $\gamma>1$ such that
    \begin{enumerate}
        \item $0\le\varphi_{R}\le1$,
        \item $\varphi_{R}\equiv1$ on $B_{R}$,
        \item $\varphi_{R}\equiv 0$ on $M\setminus B_{\gamma R}$,
        \item $\vert \nabla\varphi_{R}\vert\le\frac{C}{R}$,
        \item $\vert \Delta_{X}\varphi_{R}\vert\le\frac{C}{R}$.
    \end{enumerate}
\end{theorem}

\noindent For the proof, we need the following lemma, which is an analogue of Proposition 4.3 of \cite{ShoenYauLecturesDG} in the case of lower bounds on the Bakry–Émery-Ricci curvature.
\begin{lemma}\label{lower bound on volume}
     Let $(M,g)$ be a complete Riemannian manifold of dimension $n$ and fix a point $o\in M$. Let $X$ be a smooth vector field such that $\vert X\vert \le \Lambda$ and $\mathrm{Ric}_{X}\ge-(n-1)K$ for some $\Lambda,K\ge0$. There exists a constant $C=C(n,K,\Lambda)>0$ such that
     \begin{equation*}
         \vert B_{1}(x)\vert \ge \vert B_{1}(o)\vert e^{-Cr(x)} \quad\quad \forall x\in M.
     \end{equation*}
\end{lemma}
\begin{proof}
    Let $x\in M$, set $\rho(y):=d_{g}(x,y)$ and consider the function $F: [0,+\infty)\rightarrow [0,+\infty)$ given by
    \begin{equation*}
        F(t):= t^{-n}e^{-Ct}\vert B_{t}(x)\vert,
    \end{equation*}
    where $C>0$ is a constant that we will choose later. We want to prove that $F$ is non-increasing. First, note that from the comparison principle in Theorem \ref{X-comparison for vector field} we obtain
    \begin{equation*}
        \begin{split}
            \frac{1}{2}\Delta_{X}\rho^{2} &= \rho\Delta_{X}\rho + \vert\nabla\rho\vert^{2}\\
            &\le\rho\bigg(\frac{n-1}{\rho}+(n-1)\sqrt{K}+\Lambda\bigg)+1 \\
            &=\big((n-1)\sqrt{K}+\Lambda\big)\rho+n \\
            &=:C_{1}\rho+n
        \end{split}
    \end{equation*}
    pointwise away from the cut locus of $x$ and in the global distributional sense. Thus,
    \begin{equation*}
        \frac{1}{2}\int_{B_{t}(x)}\Delta_{X} \rho^{2}\le (n + C_{1} t)\vert B_{t}(x)\vert,
    \end{equation*}
    and also
    \begin{equation*}
        \begin{split}
            \frac{1}{2}\int_{B_{t}(x)}\Delta_{X} \rho^{2} &= \frac{1}{2}\int_{\partial B_{t}(x)}g(\nabla \rho^{2},\nabla\rho) -\frac{1}{2}\int_{B_{t}(x)}g(X,\nabla\rho^{2})\\
            &\ge t\vert\partial B_{t}(x)\vert -t\Lambda\vert B_{t}(x)\vert.
        \end{split}
    \end{equation*}
    Combining the two estimates and introducing the quantity $V(t):=\vert B_{t}(x)\vert$, by the coarea formula we deduce
    \begin{equation*}
        \begin{split}
            tV'(t)\le (C_{2}t+n)V(t),
        \end{split}
    \end{equation*}
    where $C_{2}:=C_{1}+\Lambda = (n-1)\sqrt{K}+2\Lambda$. The last expression can be rewritten as
    \begin{equation*}
        \frac{\mathrm{d}}{\mathrm{d}t}(t^{-n}e^{-C_{2}t}V(t))\le0,
    \end{equation*}
    which means that $F$ is non-increasing with the choice $C=C_2$. In particular, it follows that
    \begin{equation*}
        \begin{split}
            e^{-C_{2}}\vert B_{1}(x)\vert&\ge F(r(x)+1)\\
            &=(1+r(x))^{-n}e^{-C_{2}(r(x)+1)}\vert B_{r(x)+1}(x)\vert,
        \end{split}
    \end{equation*}
    hence, since $B_{1}(o)\subseteq B_{1+r(x)}(x)$,
    \begin{equation*}
        \begin{split}
            \vert B_{1}(x)\vert&\ge(1+r(x))^{-n}e^{-C_{2}r(x)}\vert B_{r(x)+1}(x)\vert \\
            &\ge\frac{\vert B_{1}(o)\vert}{(1+r(x))^{n}}e^{-C_{2}r(x)} \\
            &\ge \vert B_{1}(o)\vert e^{-C_{3}r(x)}e^{-C_{2}r(x)}=:\vert B_{1}(o)\vert e^{-C_4 r(x)}
        \end{split}
    \end{equation*}
    for some constant $C_{4}>0$ depending only on $n, K$ and $\Lambda$.
\end{proof}

\begin{proof}[Proof of Theorem \ref{existence of X-Lap cutoff}]
    The proof exploits a strategy due to Schoen and Yau (see \cite{ShoenYauLecturesDG}, and also \cite{BS18}) and is divided in four steps. In the first three steps, we construct a proper function $F\in C^{\infty}(M)$ such that
    \begin{enumerate}
        \item $F(x)\ge C r(x)$,
        \item $\vert\nabla F\vert\le C$,
        \item $\vert\Delta_{X} F\vert\le C$,
    \end{enumerate}
    for some constant $C>0$. In the last step, we use this function to construct the desired $X$-Laplacian cut off function.
    
    \textbf{Step 1.} Let $o\in M$ and $\lambda>0$ be a constant. For every $R>1$, consider the solution to the Dirichlet problem
    \begin{equation}\label{sol Dir probl - X-lap cutoff fct}
        \begin{cases}
            \Delta_{X} f_{R}=\lambda f_{R} &\text{on $B_{R}\setminus \overline{B_{1}}$,}\\
            f_{R}=0 &\text{on $\partial B_{R}$,}\\
            f_{R}=1 &\text{on $\partial B_{1}$.}
        \end{cases}
    \end{equation}
    From standard regularity theory we know that every $f_{R}$ is smooth in $B_{R}\setminus \overline{B}_{1}$, and by the maximum principle for the $X$-Laplacian we have
    \begin{equation*}
        0\le f_{R_{1}}(x)\le f_{R_{2}}(x) \le 1 
    \end{equation*}
    for every $x\in B_{R_{1}}\setminus\overline{B_{1}}$ and $R_{2}\ge R_{1}>1$.
    By monotonicity, we have the pointwise limit
    \begin{equation*}
        f(x):=\lim_{R\rightarrow+\infty}f_{R}(x)\in (0,1] \quad \forall x\in M.
    \end{equation*}
    By the classical $L^{p}$ and Schauder estimates (see \cite{gilbarg1998elliptic}) it follows that for every $k\ge0$ and $R_{0}>1$ there exists a constant $\tilde{C}>0$ depending only on the metric $g$, $\Lambda$, $R_{0}$, $n$, $K$ and $k$ such that
    \begin{equation*}
        \Vert f_{R}\Vert_{C^{k+2,\alpha}(\overline{B_{R_{0}}\setminus B_{1}})}\le \tilde{C} \quad\quad\forall R>R_{0}.
    \end{equation*}
    The last family of estimates imply that the sequence $\{f_{R}\}_{R>1}$ converges (up to subsequences) to $f$ in the $C^{\infty}_{loc}(M)$ topology. 
    In particular, the function $f$ is smooth and satisfies
    \begin{equation}\label{sys for f - X-lap cutoff fct}
        \begin{cases}
            \Delta_{X} f=\lambda f &\text{on $M\setminus \overline{B_{1}}$,}\\
            f=1 &\text{on $\partial B_{1}$,} \\
            0< f< 1 &\text{on $M\setminus \overline{B_{1}}$.}
        \end{cases}
    \end{equation}

    \textbf{Step 2.} Let $C>0$ be a constant and let $\nu=-\nabla r$ be the outward unit normal to $\partial B_{1}$, then
    \begin{equation*}
        \begin{split}
            \int_{B_{R}\setminus B_{1}}e^{Cr}f_{R}\Delta_{X}f_{R} &= \int_{\partial B_{1}}e^{Cr}f_{R}\frac{\partial f_{R}}{\partial\nu} -  \int_{B_{R}\setminus B_{1}}g(\nabla(e^{Cr}f_{R}),\nabla f_{R}) \\
            &- \int_{B_{R}\setminus B_{1}}e^{Cr}f_{R}g(X,\nabla f_{R}) \\
            &= e^{C}\int_{\partial B_{1}}\frac{\partial f_{R}}{\partial\nu} -  C\int_{B_{R}\setminus B_{1}}e^{Cr}f_{R}g(\nabla r,\nabla f_{R}) - \int_{B_{R}\setminus B_{1}}e^{Cr} \vert \nabla f_{R}\vert^{2} \\
            &- \int_{B_{R}\setminus B_{1}}e^{Cr} f_{R}g(X,\nabla f_{R})\\
            &\le e^{C}\tilde{C}\vert \partial B_{1}\vert +  (C+\Lambda)\int_{B_{R}\setminus B_{1}}e^{Cr}f_{R}\vert\nabla f_{R}\vert - \int_{B_{R}\setminus B_{1}}e^{Cr} \vert \nabla f_{R}\vert^{2},
        \end{split}
    \end{equation*}
    where $\tilde{C}$ comes from the Schauder estimate in the previous step. Having set $C_{1}:=e^{C}\tilde{C}\vert \partial B_{1}\vert$ and using the equation satisfied by $f_{R}$, we get
    \begin{equation*}
        \begin{split}
            \lambda \int_{B_{R}\setminus B_{1}}e^{Cr}f_{R}^{2} &\le C_1 +  (C+\Lambda)\int_{B_{R}\setminus B_{1}}e^{Cr}f_{R}\vert\nabla f_{R}\vert - \int_{B_{R}\setminus B_{1}}e^{Cr} \vert \nabla f_{R}\vert^{2}\\
            &\le C_1 +  \frac{(C+\Lambda)^{2}}{4}\int_{B_{R}\setminus B_{1}}e^{Cr}f_{R}^{2},
        \end{split}
    \end{equation*}
    and so
    \begin{equation*}
        \bigg(\lambda - \frac{(C+\Lambda)^{2}}{4}\bigg) \int_{B_{R}\setminus B_{1}}e^{Cr}f_{R}^{2}\le C_1.
    \end{equation*}
    Choosing $\lambda \ge \frac{(C+\Lambda)^{2}}{4}+1$ and letting $R\rightarrow+\infty$, the last estimate implies
    \begin{equation*}
         \int_{B_{R}\setminus B_{1}}e^{Cr}f^{2}\le C_{1}.
    \end{equation*}
    Let $x\in M\setminus \overline{B}_{2}$ and $y\in B_{1}(x)$, for $R$ sufficiently big we have
    \begin{equation*}
        C_1\ge \int_{B_{R}\setminus B_{1}}e^{Cr}f^{2}\ \mathrm{dv}(y)> \int_{B_{1}(x)}e^{Cr}f^{2}\ \mathrm{dv}(y) \ge e^{C(r(x)-1)}\int_{B_{1}(x)}f^{2}\ \mathrm{dv}(y), 
    \end{equation*}
    hence
    \begin{equation}\label{stima su anello B_R meno B_1}
        \int_{B_{1}(x)}f^{2}\ \mathrm{dv}(y)\le C_{1}e^{-C(r(x)-1)}.
    \end{equation}
    From the Harnack inequality \eqref{formula of X-Harnack} with $F(t)= \lambda t$, we have
    \begin{equation*}
        f(y)\ge C_{2} f(x),
    \end{equation*}
    where $C_{2}:=e^{-C(n,K,\lambda,\Lambda)}$ and $C(n,K,\lambda,\Lambda)$ is the constant in Corollary \ref{X-Harnack}. Using \eqref{stima su anello B_R meno B_1} we obtain
    \begin{equation}\label{estimate on f with B_1}
        C_{2}^2\vert B_{1}(x)\vert f^{2}(x)\le C_{1}e^{-C(r(x)-1)},
    \end{equation}
    and by Lemma \ref{lower bound on volume} we get 
    \begin{equation}\label{lower bound volume in X-cutoff}
        \vert B_{1}(x)\vert \ge \vert B_{1}(o)\vert e^{-C_{3} r(x)},
    \end{equation}
    where $C_{3}=C(n,K,\vert B_{1}(o)\vert,\Lambda)$. Choosing $C>C_{3}$ and combining \eqref{estimate on f with B_1} and \eqref{lower bound volume in X-cutoff}, we have
    \begin{equation*}
        C_{2}^2\vert B_{1}(o)\vert e^{-C_{3}r(x)} f^{2}(x)\le C_{1}e^{-C(r(x)-1)},
    \end{equation*}
    hence 
    \begin{equation}\label{stima dall'alto f - comstr exhaus fct}
        f(x)\le C_4 e^{-C_{5}r(x)}
    \end{equation}
    for some constants $C_{4},C_{5}>0$.

    \textbf{Step 3.} Let $\chi\in C^{\infty}(M)$ be a cutoff function such that
    \begin{enumerate}
        \item $0\le\chi \le1$ on $M$,
        \item $\chi\equiv 0$ on $M\setminus B_{2}$,
        \item $\chi\equiv 1$ on $\overline{B}_{1}$,
    \end{enumerate}
    and define $F(x) := -(1-\chi(x))\log f(x) + \chi(x)\in C^{\infty}(M)$. From \eqref{stima dall'alto f - comstr exhaus fct} we have 
    \begin{equation*}
        F(x) = -\log f(x) \ge C_{5}r(x) - \log C_{4}\ge C_{5}r(x)
    \end{equation*}
    on $M\setminus B_{2}$, and so $F$ is proper. Moreover, on $M\setminus B_{2}$ we have
    \begin{equation*}
        \vert \nabla F\vert  =\vert \nabla \log f\vert \le C_{6}
    \end{equation*}
    by the gradient estimate \eqref{gradient estimate general X}, and
    \begin{equation*}
        \Delta_{X} F  =-\frac{\Delta_{X} f}{f} + \vert \nabla \log f\vert^{2} = -\lambda + \vert \nabla \log f\vert^{2},
    \end{equation*}
    which implies
    \begin{equation*}
        \vert \Delta_{X} F \vert  \le \lambda + \vert \nabla \log f\vert^{2}\le C_{7}.
    \end{equation*}
    From the compactness of $B_{2}$ and the smoothness of $F$ we can extend the previous estimates globally on $M$ and for suitable constants.

    \textbf{Step 4.} Let $R\ge2$, to construct the $X$-Laplacian cutoff function $\varphi_{R}$ it is enough to consider the function $\varphi_{R}:=\psi(\frac{F(x)}{R})$, where $\psi\in C^{\infty}(\mathbb{R})$ satisfies
    \begin{enumerate}
        \item $0\le \psi\le1$,
        \item $\psi(t)= 1$ for $t\le 1$,
        \item $\psi(t)= 0$ for $t\ge 2$.
    \end{enumerate}
    The desired properties follow immediately by using the properness and the estimates on $F$ we found in Step 3.
    \end{proof}

\begin{remark}
    We stress that a gradient estimate like the one in \cite{brighton2013liouville, munteanu2011smooth} is not enough for the validity of the properties in Step 3, because it does not provide the necessarily pointwise bound on $\vert \nabla\log f\vert$. Thus, the extension \eqref{global gradient estimate for X-Lap} of the Cheng-Yau gradient estimates to the equation $\Delta_{X} u = \lambda u$, $\lambda>0$, is necessarily.
\end{remark}

\subsection{$L^{p}$-Liouville theorem}
Thanks to the existence of $X$-Laplacian cutoff functions, we can prove the following $L^p$-Liouville theorem for non negative $X$-subharmonic functions which are $L^p$-integrable with respect to a general invariant measure for $\Delta_{X}$.

\begin{theorem}\label{LpLiou_invariant_measure}
    Let $(M,g)$ be a complete Riemannian manifold and $p\in(1,+\infty)$. Let $X$ be a smooth vector field such that $\vert X\vert\le\Lambda$ and $\mathrm{Ric}_{X}\ge-(n-1)K$ for some $\Lambda, K\ge0$. If $\mu$ is a positive Radon measure which is an invariant measure for $\Delta_{X}$, then any function $ u\in C^{\infty}(M)$ satisfying
    \begin{equation*}
        \begin{cases}
            \Delta_{X}u\ge 0,\\
            0\le u\in L^{p}(M,\mu)
        \end{cases}
    \end{equation*}
    is constant. 
\end{theorem}

\begin{remark}
    As we will see in the proof of Theorem \ref{LpLiou_invariant_measure}, the result is true for non-negative $X$-subharmonic functions $u$ satisfying
    \begin{equation*}
        \int_{B_{\gamma R}\setminus B_{R}}u^{p}\ \mathrm{d}\mu = o(R) \quad\quad \text{as $R\rightarrow +\infty$.}
    \end{equation*}
\end{remark}
\begin{proof}
    Let $u\in C^{\infty}(M)$ be a non-negative, smooth solution to
    \begin{equation*}
        \Delta_{X} u\ge0,
    \end{equation*}
    then the function $v_{\delta}:=(u+\delta)^{p}$, $\delta>0$, is smooth, strictly positive and satisfies
    \begin{equation*}
        \begin{split}
            \Delta_{X} v_{\delta}&= p(u+\delta)^{p-1}\Delta_{X}u + p(p-1)(u+\delta)^{p-2}\vert\nabla u\vert^{2}\\
            &\ge p(p-1)(u+\delta)^{p-2}\vert\nabla u\vert^{2}.
        \end{split}
    \end{equation*}
    From the invariance of $\mu$, for every $\varphi\in C^{\infty}_{c}(M)$ we have
\begin{equation*}
    \begin{split}
        0&=\int_{M}\Delta_{X}(v_{\delta}\varphi^{2})\ \mathrm{d}\mu \\
        &=\int_{M}\varphi^{2}\Delta_{X}v_{\delta}\ \mathrm{d}\mu + \int_{M}v_{\delta}\Delta_{X}(\varphi^{2})\ \mathrm{d}\mu + 2 \int_{M}g(\nabla\varphi^{2}, \nabla v_{\delta})\ \mathrm{d}\mu \\
        &\ge p(p-1)\int_{M}(u+\delta)^{p-2}\vert\nabla u\vert^{2}\varphi^{2}\ \mathrm{d}\mu + 2\int_{M}(u+\delta)^{p}\varphi\Delta_{X}\varphi\ \mathrm{d}\mu + 2\int_{M}(u+\delta)^{p}\vert\nabla\varphi\vert^{2}\ \mathrm{d}\mu \\
        &+ 4p\int_{M}(u+\delta)^{p-1}\varphi g(\nabla\varphi, \nabla u)\ \mathrm{d}\mu,
    \end{split}
\end{equation*}
hence
\begin{equation*}
    \begin{split}
         p(p-1)\int_{M}(u+\delta)^{p-2}\vert\nabla u\vert^{2}\varphi^{2}\ \mathrm{d}\mu&\le -2\int_{M}(u+\delta)^{p}\varphi\Delta_{X}\varphi\ \mathrm{d}\mu - 2\int_{M}(u+\delta)^{p}\vert\nabla\varphi\vert^{2}\ \mathrm{d}\mu \\
        & -4p\int_{M}(u+\delta)^{p-1}\varphi g(\nabla\varphi, \nabla u)\ \mathrm{d}\mu \\
        &\le2\int_{M}(u+\delta)^{p}\vert \varphi\Delta_{X}\varphi\vert\ \mathrm{d}\mu + 2\int_{M}(u+\delta)^{p}\vert\nabla\varphi\vert^{2}\ \mathrm{d}\mu \\
        &+ 4p\int_{M}(u+\delta)^{p-1}\vert\varphi\vert \vert\nabla\varphi\vert \vert\nabla u\vert\ \mathrm{d}\mu.
    \end{split}
\end{equation*}
From the compactness of the support of $\varphi$ and using the Fatou and the dominated convergence theorem, we can take the liminf for $\delta\rightarrow0^{+}$ to obtain
\begin{equation*}
    \begin{split}
         p(p-1)\int_{M}u^{p-2}\vert\nabla u\vert^{2}\varphi^{2}\ \mathrm{d}\mu&\le2\int_{M}u^{p}\vert \varphi\Delta_{X}\varphi\vert\ \mathrm{d}\mu + 2\int_{M}u^{p}\vert\nabla\varphi\vert^{2}\ \mathrm{d}\mu \\
        &+ 4p\int_{M}u^{p-1}\vert\varphi\vert \vert\nabla\varphi\vert \vert\nabla u\vert\ \mathrm{d}\mu.
    \end{split}
\end{equation*}
Using Young inequality, we get
\begin{equation*}
    \begin{split}
         p(p-1)\int_{M}u^{p-2}\vert\nabla u\vert^{2}\varphi^{2}\ \mathrm{d}\mu&\le2\int_{M}u^{p}\vert \varphi\Delta_{X}\varphi\vert\ \mathrm{d}\mu + 2\int_{M}u^{p}\vert\nabla\varphi\vert^{2}\ \mathrm{d}\mu \\
        &+ 4p\int_{M}u^{p-1}\vert\varphi\vert \vert\nabla\varphi\vert \vert\nabla u\vert\ \mathrm{d}\mu \\
        &\le2\int_{M}u^{p}\vert \varphi\Delta_{X}\varphi\vert\ \mathrm{d}\mu + 2\int_{M}u^{p}\vert\nabla\varphi\vert^{2}\ \mathrm{d}\mu \\
        & +\frac{2p}{\varepsilon}\int_{M}u^{p} \vert\nabla\varphi\vert^{2}\ \mathrm{d}\mu + 2\varepsilon p\int_{M}u^{p-2}\vert\nabla u\vert^2 \varphi^{2}\ \mathrm{d}\mu.
    \end{split}
\end{equation*}
A simple algebraic manipulation gives
\begin{equation*}
    p(p-1-2\varepsilon)\int_{M}u^{p-2}\vert\nabla u\vert^{2}\varphi^{2}\ \mathrm{d}\mu
    \le2\int_{M}u^{p}\vert \varphi\Delta_{X}\varphi\vert\ \mathrm{d}\mu + 2\bigg(1+\frac{p}{\varepsilon}\bigg)\int_{M}u^{p}\vert\nabla\varphi\vert^{2}\ \mathrm{d}\mu,
\end{equation*}
and choosing $0<\varepsilon<\frac{p-1}{2}$ we can find a constant $C_{p}>0$ depending only on $p$ such that
\begin{equation*}
    \begin{split}
         \int_{M}u^{p-2}\vert\nabla u\vert^{2}\varphi^{2}\ \mathrm{d}\mu\le C_{p}\bigg(\int_{M}u^{p}\vert \varphi\Delta_{X}\varphi\vert\ \mathrm{d}\mu + \int_{M}u^{p}\vert\nabla\varphi\vert^{2}\ \mathrm{d}\mu\bigg).
    \end{split}
\end{equation*}
Choosing a sequence $\varphi_{R}\in C^{\infty}_{c}(M)$ as in Theorem \ref{existence of X-Lap cutoff}, we have
\begin{equation*}
    \begin{split}
         \int_{B_{R}}u^{p-2}\vert\nabla u\vert^{2}\ \mathrm{d}\mu\le \frac{2C_{p}C}{R}\int_{B_{\gamma R}\setminus B_{R}}u^{p}\ \mathrm{d}\mu \longrightarrow 0 \quad\quad \text{as $R\rightarrow+\infty,$}
    \end{split}
\end{equation*}
and thus $u$ is constant.
\end{proof}

\begin{remark}
     Note that the vector field in Example \ref{failure_LpLiou_notInv} is bounded and has Bakry–Émery-Ricci curvature bounded below, but the standard Riemannian measure is not invariant for the operator $\Delta_{X}$. 
     Indeed, computing the Ricci curvature of a warped product manifold we find that Ric$=-(n-1)g$. Moreover, if $Y$ is a vector field on $M$ of the form $Y=a(t,x)\partial_{t}+Z(t,\cdot)$, with $Z(t,\cdot)$ a vector field on $\mathbb{T}^{n-1}$ possibly depending on $t$, then the Lie derivative of the metric $g$ in the direction of $Y$ is
    \begin{equation*}
        \mathcal{L}_{Y}g = 2\partial_{t}a(t,x)\mathrm{d}t^{2}+2e^{2t}a(t,x)g_{\mathbb{T}^{n-1}}+e^{2t}\mathcal{L}_{Z}g_{\mathbb{T}^{n-1}} + (\mathrm{d}t\otimes \omega + \omega\otimes\mathrm{d}t), 
    \end{equation*}
    where 
    \begin{equation*}
        \omega = \mathrm{d}_{\mathbb{T}^{n-1}}a + e^{2t}g_{\mathbb{T}^{n-1}}(\partial_{t}Z,\cdot),
    \end{equation*}
    and its divergence is
    \begin{equation*}
        \mathrm{div}X=\partial_{t}a+(n-1)a+\mathrm{div}_{\mathbb{T}^{n-1}}Z.
    \end{equation*}
    In the case of $X=\alpha\partial_{t}$, we obtain $\mathrm{Ric}_{X}\ge-(n-1)$ and $\mathrm{div}X=(n-1)\alpha>0$. In particular, the measure dv is not invariant for $\Delta_{X}$, as we saw in Example \ref{Exmp_Riemm_density}. 
\end{remark}

\begin{remark}
    If we do not impose some hypothesis on the vector field or the Bakry–Émery-Ricci curvature, the $L^{p}$-Liouville property can fail for some invariant measure. Indeed, in Example 6.5 of \cite{PessoaPigolaSetti2026volumegrowthparabolicitynonsymmetric} the authors constructed an unbounded vector field $X$ for which there exists an invariant measure $\mu$ supporting a non-negative $X$-subharmonic function in $L^{p}(M,\mathrm{d}\mu)$ for every $p\in(1,+\infty)$.
\end{remark}

\noindent To conclude, it is interesting to compare Theorem \ref{LpLiou_invariant_measure} with the $L^{p}$-Liouville theorem obtained in \cite{PessoaPigolaSetti2026volumegrowthparabolicitynonsymmetric}, which is the following 
\begin{theorem}[Theorem 6.3 in \cite{PessoaPigolaSetti2026volumegrowthparabolicitynonsymmetric}]
    Let $(M,g)$ be a complete, non-compact Riemannian manifold endowed with a smooth vector field $X$, and let $p\in(1,+\infty)$. Consider a positive Radon measure $\mu = e^{-f}\mathrm{dv}$, $f\in C^{\infty}(M)$, which is an invariant measure for $\Delta_{X}$. If 
    \begin{equation*}
        \delta_{f}^{\infty}(\Delta_{X}):=\Vert X - \nabla f\Vert_{L^{\infty}(M)}<+\infty,
    \end{equation*}
    then any function $ u\in C^{\infty}(M)$ satisfying
    \begin{equation*}
        \begin{cases}
            \Delta_{X}u\ge 0,\\
            0\le u\in L^{p}(M,\mu)
        \end{cases}
    \end{equation*}
    is constant.
\end{theorem}

\noindent This result gives the $L^{p}$-Liouville property without any assumption on the growth of the vector field or the Bakry–Émery-Ricci curvature. The drawback is that the result holds only for invariant measures with finite $L^{\infty}$-defect $\delta_{f}^{\infty}(\Delta_{X})$. On the other hand, Theorem \ref{LpLiou_invariant_measure} requires a bound on the norm of $X$ and on $\mathrm{Ric}_{X}$, but is valid for $any$ invariant measure for $\Delta_{X}$.

\subsection*{Acknowledgements} The author is grateful to professors Stefano Pigola and Giona Veronelli for the several discussions and
the support during the preparation of the present work. The author is member of GNAMPA-INdAM. The author is partially supported by the 2026 INdAM-GNAMPA project \emph{Asymptotic analysis of variational problems}, n. \texttt{CUP\_E53C25002010001}. 

\ \\

\noindent \textbf{AI Statement.} No AI tools were used in the preparation of this manuscript except
for a final check for misprints and notational inconsistencies.

\appendix
\section{Examples}
In this appendix, we provide some explicit examples of Riemannian manifolds $(M,g)$ supporting a smooth, not necessarily gradient, vector field $X$ such that its norm decays to zero at infinity and with non-negative Bakry-\'Emery-Ricci curvature, both in the gradient and non-gradient cases.

\subsection{The geometric setting} 
On the Euclidean space $\mathbb{R}^{2}$, define the complete metric $g$ given by the 2-symmetric tensor that, in global coordinates $(x,y)$, is associated to the matrix
\begin{equation*}
    g = 
    \begin{pmatrix}
       1+4x^{2} & 4xy\\
       4xy & 1+4y^{2}
    \end{pmatrix}.
\end{equation*}
Note that we can identify $(\mathbb{R}^{2},g)$ with the paraboloid $\Sigma:=\{(x,y,z)\in\mathbb{R}^{3}\mid z=x^{2}+y^{2}\}\subseteq\mathbb{R}^{3}$ endowed with the metric induced by the inclusion $\iota:\Sigma\hookrightarrow\mathbb{R}^{3}$. With this identification and from the fact that $\Sigma$ is the graph of the function $f(x,y)=x^{2}+y^{2}$, we deduce that the Gauss curvature of $(\mathbb{R}^{2},g)$ is 
\begin{equation*}
    K(x,y)=\frac{f_{xx}f_{yy}-f_{xy}^{2}}{(1+f_{x}^{2}+f_{y}^{2})^{2}}=\frac{4}{(1+4x^{2}+4y^{2})^{2}},
\end{equation*}
so the Ricci curvature can be written as 
\begin{equation*}
    \mathrm{Ric}=Kg=\frac{4}{(1+4x^{2}+4y^{2})^{2}}
    \begin{pmatrix}
       1+4x^{2} & 4xy\\
       4xy & 1+4y^{2}
    \end{pmatrix}.
\end{equation*}
Moreover, the inverse of the metric $g$ is the tensor expressed by the matrix
\begin{equation*}
    g^{-1} = \frac{1}{1+4x^{2}+4y^{2}}
    \begin{pmatrix}
       1+4y^{2} & -4xy\\
       -4xy & 1+4x^{2}
    \end{pmatrix}.
\end{equation*}
Denote the coordinates of $\mathbb{R}^{2}$ by $x^{1}=x$ and $x^{2}=y$ , then the Christoffel symbols are given by the formula
\begin{equation*}
    \Gamma^{k}_{ij}=\frac{1}{2}g^{kt}\bigg(\frac{\partial g_{it}}{\partial x^{j}}+\frac{\partial g_{jt}}{\partial x^{i}}-\frac{\partial g_{ij}}{\partial x^{t}}\bigg),
\end{equation*}
and a long computation shows that
\begin{align*}
    &\Gamma^{1}_{11}=\Gamma^{1}_{22}=\frac{4x}{1+4x^{2}+4y^{2}}, \\
    &\Gamma^{1}_{12}=\Gamma^{1}_{21}=0, \\
    &\Gamma^{2}_{11}=\Gamma^{2}_{22}=\frac{4y}{1+4x^{2}+4y^{2}}, \\
    &\Gamma^{2}_{12}=\Gamma^{2}_{21}=0. \\
\end{align*}
\begin{notation}
    In what follows, we will use the symbols $\overline{\nabla}$, $\overline{\mathrm{Hess}}$, and so on to denote the geometric and analytic quantities related to the standard flat metric on $\mathbb{R}^{2}$. Moreover, if $p\in\mathbb{R}^{2}$ and $R>0$, we denote by $\mathbb{B}_{R}(p)$ the open Euclidean ball of center $p$ and radius $R$; if the center is the origin, we use the simplified notation $\mathbb{B}_{R}$.
\end{notation}

\subsection{A gradient example}
To construct the example in the gradient case, we define the smooth function $\Phi\in C^{\infty}(\Sigma)$ as
\begin{equation*}
    \Phi(x,y):=\frac{1}{8}\log(1+4x^{2}+4y^{2}),
\end{equation*}
then the Euclidean gradient is 
\begin{equation*}
    \overline\nabla\Phi = \frac{(x,y)}{1+4x^{2}+4y^{2}}
\end{equation*}
and the Euclidean Hessian is
\begin{equation*}
    \overline{\mathrm{Hess}}(\Phi)=\frac{1}{(1+4x^{2}+4y^{2})^{2}}
    \begin{pmatrix}
       1-4x^{2}+4y^{2} & -8xy\\
       -8xy & 1+4x^{2}-4y^{2}
    \end{pmatrix}.
\end{equation*}
This implies that
\begin{equation*}
    \begin{split}
        \mathrm{Hess}(\Phi) &=  \overline{\mathrm{Hess}}(\Phi) - \bigg(\Gamma_{ij}^{k}\frac{\partial\Phi}{\partial x^{k}}\bigg)_{i,j=1,2} \\
        &=\frac{1}{(1+4x^{2}+4y^{2})^{2}}\begin{pmatrix}
       1-8x^{2} & -8xy\\
       -8xy& 1-8y^{2}
    \end{pmatrix},
    \end{split}
\end{equation*}
hence
\begin{equation*}
    \begin{split}
        \mathrm{Ric}_{\Phi} = \mathrm{Ric}+\mathrm{Hess}(\Phi) =\frac{4}{(1+4x^{2}+4y^{2})^{2}}\begin{pmatrix}
        1+2x^{2}& 2xy\\
        2xy& 1+2y^{2}
    \end{pmatrix},
    \end{split}
\end{equation*}
which is a non-negative matrix at every point. Moreover, the norm of the gradient of $\Phi$ is
\begin{equation*}
    \begin{split}
        \vert\nabla\Phi\vert^{2}&=g(\nabla\Phi,\nabla\Phi)=(\overline\nabla\Phi)^{T} g^{-1}\overline\nabla\Phi \\
        &= \frac{1}{4}\frac{1}{(1+4x^{2}+4y^{2})^{3}}(x,y)
        \begin{pmatrix}
        1+4y^{2} & -4xy\\
        -4xy& 1+4x^{2}
    \end{pmatrix}
    \begin{pmatrix}
        x \\
        y
    \end{pmatrix} \\
    &=\frac{1}{4}\frac{x^{2}+y^{2}}{(1+4x^{2}+4y^{2})^{3}}
    \end{split}
\end{equation*}
hence 
\begin{equation*}
    \lim_{(x,y)\rightarrow\infty}\vert\nabla\Phi\vert=0.
\end{equation*}

\subsection{A non-gradient example through perturbation of a gradient vector field}
To obtain a non-gradient vector field $X$ satisfying the desired properties, we can proceed by modifying $\nabla\Phi$ in the following way: fix $R>0$, consider a cutoff function $\eta\in C^{\infty}(\mathbb{R}^{2})$ such that
\begin{itemize}
    \item $0\le\eta\le1$,
    \item $\eta\equiv 1$ on $\mathbb{B}_{R}$,
    \item $\eta\equiv 0$ on $\mathbb{R}^{2}\setminus\mathbb{B}_{2R}$,
\end{itemize}
and consider the smooth vector field $Y$ on $\mathbb{R}^{2}$ given by  
\begin{equation*}
    Y:=(-y,x).
\end{equation*}
We show that $Y$ is a non-gradient vector field, by proving the following
\begin{claim}
    There exists no Euclidean ball $\mathbb{B}_{r}(p)$ such that the vector field $Y$ defined above is the gradient of a smooth function $F\in C^{\infty}(\mathbb{B}_{r}(p))$.
\end{claim}
\begin{proof}
    Fix an Euclidean ball $\mathbb{B}_{r}(p)$ and note that $Y=\nabla F$ for some smooth function on $\mathbb{B}_{r}(p)$ if and only if $Y^{\flat}=\mathrm{d}F$, where $\cdot^{\flat}$ denotes the musical isomorphism induced by the metric $g$. Since $\mathbb{B}_{r}(p)$ is simply connected, then the first De Rham cohomology group is trivial, so the 1-form $\omega:=Y^{\flat}$ is exact if and only if it is closed, i.e.
\begin{equation*}
    \frac{\partial \omega_{1}}{\partial y}=\frac{\partial\omega_{2}}{\partial x}.
\end{equation*}
By definition of $ \cdot^{\flat}$, we have that the components of $\omega$ are
\begin{equation*}
    \begin{pmatrix}
        \omega_{1} \\
        \omega_{2}
    \end{pmatrix}
    =\begin{pmatrix}
        1+4x^{2} & 4xy\\
        4xy& 1+4y^{2}
        \end{pmatrix}
         \begin{pmatrix}
        -y \\
        x
        \end{pmatrix}
        =\begin{pmatrix}
        -y \\
        x
    \end{pmatrix},
\end{equation*}
hence
\begin{equation*}
    \frac{\partial\omega_{1}}{\partial y} = -1, \quad \frac{\partial\omega_{2}}{\partial x} = 1.
\end{equation*}
Since $\frac{\partial\omega_{1}}{\partial y}\ne\frac{\partial\omega_{2}}{\partial x}$, we deduce that $\omega$ is not a closed 1-form, and then $Y$ is not the gradient of a smooth function on $\mathbb{B}_{r}(p)$.
\end{proof}
\noindent Now, define the smooth vector field $X:=\nabla\Phi+\varepsilon\eta Y$, with $\varepsilon>0$ small to be chosen. Since $\eta$ has compact support, then $\vert X\vert =\vert\nabla\Phi\vert$ at infinity, so
\begin{equation*}
    \lim_{(x,y)\rightarrow\infty}\vert X\vert=0.
\end{equation*}
Moreover, $X$ is not a gradient vector field: indeed, if there exists $F\in C^{\infty}(\mathbb{R}^{2})$ such that $X=\nabla F$, then 
\begin{equation*}
    \eta Y=\nabla (\Phi-F).
\end{equation*}
In particular, we would have $Y=\nabla(\Phi-F)$ on $\mathbb{B}_{R}$, and this contradicts the claim.

Finally, since $\mathrm{Ric}+\mathrm{Hess}(\Phi)>0$ on the Euclidean ball $\mathbb{B}_{2R}$, by continuity we can choose $\varepsilon$ small enough to ensure that $\mathrm{Ric}+\mathrm{Hess}(\Phi)>\mathcal{L}_{\varepsilon\eta Y}g$ in the sense of the quadratic forms, and this implies that $\mathrm{Ric}_{X}\ge0$ on $\mathbb{R}^{2}$.
 
\subsection{Another example in the non-gradient case}
Another way to produce a non-gradient vector field with all the desired properties is by multiplication of a non-gradient vector field for a suitable smooth function. Consider again the vector field $Y=(-y,x)$, and set 
\begin{align*}
    A_{0} &:= \mathbb{B}_{1}, \\
    A_{n}&:=\{(x,y)\in\mathbb{R}^{2}\mid n\le x^{2}+y^{2}\le n+1\}, \quad \forall n\in\mathbb{N}\setminus\{0\}.
\end{align*}
Fix a small $\delta>0$, and for every $n\ge0$ pick a cutoff function $\eta_{n}\in C^{\infty}(\mathbb{R}^{2})$ such that
\begin{itemize}
    \item $0\le\eta_{n}\le1$;
    \item $\eta_{n}\equiv 1$ on $\{(x,y)\in\mathbb{R}^{2}\mid n+\delta \le x^{2}+y^{2} \le n+1-\delta\}$ if $n\ge1$, and on $\{(x,y)\in\mathbb{R}^{2}\mid x^{2}+y^{2} \le 1-\delta\}$ if $n=0$;
    \item the support of $\eta_{n}$ is contained in the interior of $A_{n}$.
\end{itemize}
For every $n\ge0$, we choose $\varepsilon_{n}>0$ such that
\begin{equation*}
    \varepsilon_{n}\le\frac{\min_{A_{n}}K}{\max_{A_{n}}\vert\mathcal{L}_{\eta_{n}Y}g\vert},
\end{equation*}
and we define the smooth function
\begin{equation*}
    f:=\sum_{n=0}^{+\infty}\varepsilon_{n}\eta_{n},
\end{equation*}
which is well defined since the support of the cutoff functions are disjoint. Finally, if we consider the smooth vector field $X:=f Y$, then we can check that $X$ is non-gradient and satisfies all the desired properties.

\bibliographystyle{alpha}
\bibliography{bibliography}

\end{document}